\documentclass[12pt,reqno]{article}
\usepackage{fullpage}
\usepackage{comment}
\usepackage{enumerate}
\usepackage{thmtools}
\usepackage{mathtools}
\usepackage[normalem]{ulem}
\usepackage{appendix}
\usepackage{amssymb}
\usepackage{amsthm}
\usepackage[normalem]{ulem}
\usepackage[mathscr]{euscript}
\usepackage{pgf,tikz}
\usetikzlibrary{arrows}
\usepackage{float}
\usepackage{hyperref}
\usepackage{cite}
\usepackage{color}
\usepackage{bbm}
\usepackage{xcolor}
\usepackage{transparent}
\usepackage{subfigure}
\usepackage{upgreek} % for $\upnu$, to distinguish $\nu$ from $v$.
\usepackage{amsfonts,txfonts}
\usepackage{bm}
\usepackage{graphicx}                       %Christian added this! idk what the dvips option is for so i just added without it
\usepackage{epstopdf}
\usepackage{psfrag}
\definecolor{azure(colorwheel)}{rgb}{0.0, 0.5, 1.0}
\definecolor{hanpurple}{rgb}{0.32, 0.09, 0.98}
\definecolor{iris}{rgb}{0.35, 0.31, 0.81}
\definecolor{byzantine}{rgb}{0.74, 0.2, 0.64}
\definecolor{ashgrey}{rgb}{0.7, 0.75, 0.71}
\definecolor{battleshipgrey}{rgb}{0.52, 0.52, 0.51}
\hypersetup{colorlinks=true,pdfborder=000,  citecolor  = blue, citebordercolor= {magenta}}
\hypersetup{linkcolor=black}

\makeatletter
\let\reftagform@=\tagform@
\def\tagform@#1{\maketag@@@{(\ignorespaces\textcolor{purple}{#1}\unskip\@@italiccorr)}}
\renewcommand{\eqref}[1]{\textup{\reftagform@{\ref{#1}}}}
\makeatother
\makeatletter
\DeclareUrlCommand\ULurl@@{%
  \def\UrlLeft{\uline\bgroup}%
  \def\UrlRight{\egroup}}
\def\ULurl@#1{\hyper@linkurl{\ULurl@@{#1}}{#1}}
\DeclareRobustCommand*\ULurl{\hyper@normalise\ULurl@}
\makeatother

\def\lessim{\ \lower4pt\hbox{$
		\buildrel{\displaystyle <}\over\sim$}\ }
\def\gessim{\ \lower4pt\hbox{$\buildrel{\displaystyle >}
		\over\sim$}\ }

\numberwithin{equation}{section}

\newcommand{\R}{\mathbb{R}}
\newcommand{\E}{\mathbb{E}}

\newcommand{\Prob}{\mathbb{P}}

\newcommand{\Z}{\mathbb{Z}}

\newcommand{\pair}[2]{\left\langle #1, #2 \right\rangle}

\newcommand{\proj}{\mathrm{proj}}
\newcommand{\tendsto}[2]{\xrightarrow[#1 \to #2]{}}

\newcommand{\ind}{\mathbbm{1}}

\newcommand{\g}{\mathfrak{g}}
\newcommand{\scl}{\mathrm{scl}}
\newcommand{\diam}{\mathrm{diam}}
\newcommand{\Aut}{\mathrm{Aut}}

\newtheorem{lemma}{Lemma}

\newtheorem{thm}{Theorem}
\newtheorem{rmk}{Remark}
\newtheorem{prop}{Proposition}

\graphicspath{ {./tikzitdiagrams/figures/} }

\begin{document}

\author{Antonio Auffinger \thanks{Department of Mathematics, Northwestern University, tuca@northwestern.edu, research partially supported by NSF Grant CAREER DMS-1653552 and NSF Grant DMS-1517894.} \\
	\small{Northwestern University}\and Christian Gorski \thanks{Department of Mathematics, Northwestern University, Email: christiangorski2022@u.northwestern.edu, research partially supported by NSF Grant DMS-1502632
RTG: Analysis on Manifolds.}\\
	\small{Northwestern University}}
\title{Asymptotic shapes for stationary first passage percolation on virtually nilpotent groups}

\maketitle
\begin{abstract}
We study first passage percolation (FPP) with stationary edge weights on Cayley graphs of finitely generated virtually nilpotent groups.
Previous works of Benjamini-Tessera \cite{BenjaminiTessera} and Cantrell-Furman \cite{CantrellFurman} show that scaling limits of such FPP are given by 
Carnot-Carath\'eodory metrics on the associated graded nilpotent Lie group.
We show a converse, i.e. that for any Cayley graph of a finitely generated nilpotent group,
any Carnot-Carath\'eodory metric on the associated graded nilpotent Lie group is the scaling limit of some FPP with stationary
edge weights on that graph.
Moreover, for any Cayley graph of any finitely generated \emph{virtually} nilpotent group, 
any ``conjugation-invariant'' metric  is the scaling limit of some 
FPP with stationary edge weights on that graph.
We also show that the ``conjugation-invariant'' condition is also a necessary condition in all cases where scaling limits are known to exist.

\end{abstract}

\section{Introduction}
\subsection{Main result}

First passage percolation (FPP) was introduced by Hammersley and Welsh \cite{HW} in 1965  as a model for the spread of a fluid through a porous medium. It is a random perturbation of a given graph distance, where random lengths are assigned to edges of a fixed graph. For a survey on this model, the reader is invited to read \cite{Aspects, ADH} and the references therein. 

The most studied case is when the fixed graph is $\Z^d$ and the edge weights are i.i.d. random variables. Under suitable moment conditions on the weight distribution, one obtains the famous shape theorem of Cox and Durrett ($d=2$) \cite{CD} and Kesten ($d>2$)\cite{Aspects}: there exists a norm $\mu$ on $\mathbb R^{d}$ such that FPP on $\Z^d$ has almost surely a deterministic scaling limit given by the normed vector space $(\mathbb R^d, \mu)$. The limiting norm $\mu$ depends on the distribution of the edge weights. It is a famous open question to determine which possible metrics arise as FPP limits on $\Z^d$ with i.i.d. edge weights. In particular, it is expected that the limit unit ball should be strictly convex, ruling out trivial metrics such as $\ell_{1}$ or $\ell_{\infty}$.

In 1995, Haggstrom and Meester\cite{HaggstromMeester} showed that if the assumption of i.i.d. edge weights on $\mathbb Z^{d}$ is relaxed, some of the expected restrictions on the limit norm disappear. Precisely, they showed that for any norm $\rho$ on $\R^d$ there exist \emph{stationary} edge weights on $\Z^d$ which give a FPP model whose scaling limit is $(\R^d, \rho)$. In this paper, we explore this direction for FPP in different (non-abelian) graphs.

Benjamini and Tessera\cite{BenjaminiTessera} explored i.i.d. FPP models on Cayley graphs of a finitely generated virtually nilpotent groups. This class of groups is precisely the class of groups with polynomial growth, due to a famous theorem of Gromov, and includes the classical example of $\Z^d$.
The question of scaling limits of such groups was first answered in the deterministic setting by Pansu ~\cite{Pansu}, who proved that, for a 
large class of invariant metrics on such groups, the scaling limit is given by a Carnot-Carath\'eodory metric on a certain nilpotent Lie group. 

Benjamini and Tessera prove that, under mild conditions, an i.i.d. FPP on a nilpotent Cayley graph also has a deterministic scaling limit given
by a Carnot-Carath\'eodory metric on a nilpotent Lie group.
Later  Cantrell and Furman ~\cite{CantrellFurman} proved an analogous theorem for \emph{stationary} edge weights.
Again, in all these cases, the limit shape depends on the distribution of the edge weights, and in the i.i.d. case, restrictions on realizable metrics
are conjectured but largely unproven.

A natural question then arises, in the spirit of Haggstrom and Meester \cite{HaggstromMeester} : for stationary FPP on virtually nilpotent groups, are all possible limit shapes realizable? What are the required symmetries for the limit metric? More explicitly, given a Cayley graph of some finitely generated virtually nilpotent group and a Carnot-Carath\'eodory metric on the associated nilpotent Lie group, 
do there exist stationary edge weights which give a FPP with a scaling limit given by that Carnot-Carath\'eodory metric?
The goal of this paper is to provide an affirmative answer to this last question in the nilpotent case and to obtain a similar characterization of all limit shapes of stationary FPPs in the virtually nilpotent case. Our main theorem is the following.

\begin{thm} \label{mainthm}
Let $\Gamma$ be a finitely generated virtually nilpotent group with generating set $S$, and let $E$ be the edge set of the
corresponding Cayley graph. Let $d_{\Phi}$ be a Carnot-Carath\'eodory
metric on the associated graded Lie group $G_{\infty}$. If $\Phi$ is conjugation invariant, then there exist stationary weights
$w:E \to \R_{\ge 0}$ such that the associated metric space $(\Gamma, T)$ satisfies
\[ \left (\Gamma, \frac{1}{n}T \right) \tendsto{n}{\infty} (G_{\infty}, d_{\Phi}) \]
in the sense of pointed Gromov-Hausdorff convergence.
\end{thm}

\begin{figure}[t]
   \centering
   \includegraphics[scale=.4]{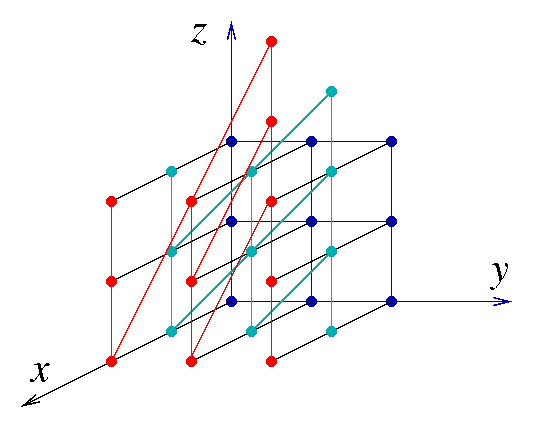}
   \caption{A portion of the Cayley graph of $H(\Z)$ with respect to the generating set $\{X,Y,Z\}$. Source: Wikipedia; image by Gabor Pete.
   Colors are for visual contrast only.}
\end{figure}

 To make the theorem more concrete, let us consider the example of the Heisenberg group, the simplest nonabelian nilpotent group.
 The integer Heisenberg group $H(\Z)$ has presentation
 \[ \langle X, Y, Z | [X,Y]=Z, [X,Z]=[Y,Z]=1 \rangle, \]
 and can be realized as the subgroup
 \[ \left\{
    \begin{bmatrix}
    1 & a & b \\
       & 1 & c \\
       &    & 1
    \end{bmatrix}
    : a,b,c \in \Z 
    \right\}
 \]
 of $GL_3(\R)$. It sits as a cocompact lattice inside the real Heisenberg group $H(\R)$, the group of real upper triangular matrices
 with 1s on the diagonal. Given any norm $\Phi$ on the subspace
 \[ V := \left\{
    \begin{bmatrix}
      & a &   \\
      &    & c \\
      &    &
    \end{bmatrix}
    :
    a,c \in \R
    \right\}
 \]
 of the Lie algebra of $H(\R)$, there exists a metric called the Carnot-Carath\'eodory metric $d_{\Phi}$ on $H(\R)$ associated to $\Phi$
 (see Appendix \ref{liegroupconstructions}). So in the special case of the Heisenberg group, our theorem is as follows:
 \begin{thm}
 Let $\Phi$ be any norm on $V$, $d_{\Phi}$ the associated Carnot-Carath\'eodory metric on $H(\R)$. Then, given any Cayley graph
 of $H(\Z)$, there exist stationary edge weights $w:E \to \R_{\ge 0}$ ($E$ the edge set of the Cayley graph) such that the resulting
 FPP metric $T$ is such that
 \[ \left(H(\Z), \frac{1}{n}T\right) \tendsto{n}{\infty} (H(\R), d_{\Phi}) \]
in the sense of pointed Gromov-Hausdorff convergence.
 \end{thm}

\subsection{Definitions, notations, and background} \label{defnsandbackground}
We now provide the definitions and the setup for Theorem \ref{mainthm}. 
Let $\Gamma$ be a finitely generated virtually nilpotent group, 
and let $S$ be a finite generating set. 
 The Cayley graph associated to $(\Gamma,S)$
 is the graph with vertex set $\Gamma$ and edge set $E:=\{ \{g,gs\} : g \in \Gamma, s \in S \}$. 
 For an element $g \in \Gamma$, set 
 \[ |g| := \inf \{ n \ge 0 : \exists s_1,...,s_n \in S \cup S^{-1} \mbox{ such that } s_1 \cdots s_n = g\}, \]
 and denote by $d$ the word metric
 \[ d(x,y) := |x^{-1} y| \]
 on $\Gamma$. Note that $d$ is a left-invariant metric on $\Gamma$.
 If $\gamma$ is an edge path in $E$, we will denote by $|\gamma|$ the number of edges in $\gamma$.
 Thus we have
 \[ d(x,y) = \inf \{ |\gamma| : \gamma \mbox{ is a path from } x \mbox{ to } y \}. \]
  
 Let $w$ be a random function $w:E \to [0,\infty)$. We call $w(e)$ the \emph{weight} of the edge $e$. 
 The collection of weights $w$ is called \emph{stationary} if the distribution is invariant under the left action of $\Gamma$, that is,
 for every finite collection of edges $f_1,...,f_k \in E$ and every $g \in \Gamma$, the joint distributions of $(w(f_1),...,w(f_k))$
 and $(w(g^{-1} f_1),..., w(g^{-1} f_k))$ are equal. The weights are called \emph{ergodic} if the underlying probability space
 is ergodic, that is, if all $\Gamma$-invariant events have probability $0$ or $1$.
 
 For an edge path $\gamma=(f_1,...,f_k)$,
 we define 
 \[ T(\gamma) := \sum_{i=1}^k w(f_i) \]
 and for two $x,y \in \Gamma$ we define the \emph{passage time} from $x$ to $y$ to be
 \[ T(x,y) := \inf \{ T(\gamma) : \gamma \mbox{ is a path from } x \mbox{ to } y \}. \]
 $T$ is a random pseudo-metric on $\Gamma$ and the pseudo-metric space $(\Gamma, T)$ is called  \emph{first passage percolation} or \emph{FPP} on $\Gamma$. Taking expectations
 we see that $\E T$ also gives a metric on $\Gamma$; if $w$ is stationary, then this metric is left-invariant.

Let $N$ be a finite index normal torsion-free nilpotent subgroup of $\Gamma$. (Such a subgroup is constructed in the course
of Appendix \ref{grouptheory}.) We denote the abelianization $N/[N,N]$ of $N$ by $N^{ab}$. This is a finitely generated abelian group, and so its torsion elements
form a finite subgroup $N^{ab}_{tor}$. We define $N^{ab}_{free} := N^{ab}/N^{ab}_{tor}$.

There is a graded nilpotent Lie group $G_{\infty}$ associated to $\Gamma$ (via $N$), and a certain subalgebra of its Lie algebra, which we denote
 by $\g^{ab}$, is equipped with a natural isomorphism $N^{ab} \otimes \R \cong \g^{ab}$. Each norm $\Psi$ on $\g^{ab}$ determines
 a metric $d_{\Psi}$ on $G_{\infty}$ which is called the Carnot-Carath\'eodory metric associated to $\Psi$; conversely, every Carnot-Carath\'eodory
 metric on $G_{\infty}$ comes from a unique norm on $\g^{ab}$. More explicit descriptions and constructions of these objects can be found in 
 Appendix \ref{liegroupconstructions}, as well as ~\cite{CantrellFurman}.

Lastly, there is a construction which plays a central role in our proof,
which associates a norm on $\g^{ab}$ to a metric on $\Gamma$.
Since $| \cdot |$ is a symmetric subadditive function on $\Gamma$ (i.e. $|ab| \le |a| + |b|$ for all $a,b \in \Gamma$), and
hence a symmetric subadditive function on $N$,
it induces a symmetric subadditive function on $N^{ab}_{free} \cong \Z^d$ via the quotient map $N \to N^{ab}_{free},
x \mapsto x^{ab}_{free}$:
\[ |y|_{ab} := \inf_{x \in N, x^{ab}_{free} = y} |x|. \]
As a symmetric subadditive function on $N^{ab}_{free} \cong \Z^d$, $|\cdot|_{ab}$ is asymptotically equivalent to a unique seminorm on 
$\R^d \cong N^{ab}_{free} \otimes \R \cong N^{ab} \otimes \R$.
That is, there is a unique seminorm $\|\cdot\|$ on $N^{ab} \otimes \R$ such that 
\[ \|y\| - |y|_{ab} = o(y) \]
where the in the little-o notation we may take any norm on $N^{ab} \otimes \R$ to measure $y$.
Similarly, assuming our weights are integrable, $\E T(1,\cdot)$ is also subadditive, 
and hence it induces a subadditive fuction $\tilde{T}$ on $N^{ab}_{free}$ which is 
asymptotically equivalent to a unique seminorm
$\Phi$ on $N^{ab} \otimes \R$.

The conjugation action of $\Gamma$ on $N$ induces an action of $\Gamma$ on $N^{ab} \otimes \R$, hence induces an action
on the set of norms on $N^{ab} \otimes \R$. We call a norm on $N^{ab} \otimes \R$ \emph{conjugation-invariant} if it is invariant
under this action. The conjugation action is discussed further in Section \ref{restrictions}, but in the case that $\Gamma$
itself is already nilpotent, the action is trivial, and hence in this case all norms on $N^{ab} \otimes \R$ are conjugation invariant.
In the Section \ref{restrictions} we also show that conjugation-invariance is a necessary restriction, that is, if $\Phi$ is a norm associated
to an invariant metric (such as $\E T$ when each $T(x,y)$ is integrable), then $\Phi$ is necessarily conjugation-invariant.

  In the notations above, it is known that $(G, d_{\| \cdot \|})$ is the scaling limit of $(\Gamma, d)$ ~\cite{Pansu} and that $(\Gamma, T)$ 
  almost surely has 
  scaling limit $(G_{\infty}, d_{\Phi})$ for many choices of edge weights \cite{BenjaminiTessera, CantrellFurman}.  Theorem \ref{mainthm} above 
  shows that any Carnot-Carath\'eodory  $d_{\Psi}$  as in \eqref{met:CC} is the scaling limit of some stationary FPP model
 on any Cayley graph of $\Gamma$, so long as $\Psi$ is conjugation-invariant.

\subsection{Proof strategy and organization of the paper}

The following theorem of Cantrell and Furman\cite{CantrellFurman}  provides a starting point for us:
 
 \begin{thm}{(\cite{CantrellFurman})} \label{cantrellfurmanthm}
 Let $w$ be ergodic stationary weights such that $T$ is bi-Lipschitz to $d$, that is, there exist $0<k<K<\infty$ such that
 \[ kd(x,y) \le T(x,y) \le Kd(x,y) \]
 for all $x,y \in \Gamma$ almost surely. Let $\Phi$ be the norm on $\g^{ab}$ associated to the metric $\E T$ on $\Gamma$, 
 and let $d_{\Phi}$ be the Carnot-Carath\'eodory metric on $G_{\infty}$ associated to $\Phi$, as above. Then almost surely
\begin{equation}\label{ehww}
 \left(\Gamma, \frac{1}{n} T, 1\right) \tendsto{n}{\infty} (G_{\infty}, d_{\Phi}, 1) 
 \end{equation}
 is the sense of pointed Gromov-Hausdorff convergence.
 \end{thm}
 \begin{rmk}
 The fact that the norm $\Phi$ we describe above is the same norm constructed in ~\cite{CantrellFurman} is perhaps not obvious
 except in the case that $\Gamma=N$ is torsion-free with torsion-free abelianization. A proof that the two constructions do give the
 same answer is given in Appendix \ref{grouptheory}.
 \end{rmk}
 \begin{rmk}
 Cantrell and Furman don't require the random metric $T$ to come from
 edge weights but require it to be \emph{inner} (see Appendix \ref{grouptheory}) in addition to being bi-Lipschitz to $d$.
 On the other hand, if $T$ comes from edge weights which are uniformly bounded above (implied by the bi-Lipschitz
 condition on $T$), then $T$ is inner, so the above statement is implied by the main theorem of ~\cite{CantrellFurman}.
 Thus our theorem shows that the collection of scaling limits of FPPs coming from stationary edge weights on a fixed Cayley graph is no smaller
 than the collection of scaling limits of stationary inner metrics which are bi-Lipschitz to $d$.
 \end{rmk}

 \begin{rmk}
In Appendix \ref{fillgap} we provide a step that was omitted in the proof of Theorem \ref{cantrellfurmanthm} in \cite{CantrellFurman}. It guarantees that the convergence in \eqref{ehww} is indeed in Gromov-Hausdorff sense. See Remark \ref{rem:correction} for more details.
 \end{rmk}
 
 In view of Theorem \ref{cantrellfurmanthm} and the correspondence between Carnot-Carath\'eodory metrics and norms on $\g^{ab}$, 
 in order to prove Theorem \ref{mainthm}, it suffices to prove:
 \begin{thm} \label{reducedthm}
 Let $\Gamma$ be a finitely generated virtually nilpotent group with generating set $S$, and let $E$ be the edge set of the
corresponding Cayley graph.
 Let $\Psi$ be a norm on $N^{ab} \otimes \R$ which is conjugation-invariant. Then there exist ergodic stationary weights $w:E \to \R$ such that
 $T$ is bi-Lipschitz to $d$, and such that the subadditive function on $N^{ab}_{free}$ induced by $\E T(1, \cdot)$
 is asymptotically equivalent to $\Psi$.
 \end{thm}

\begin{proof}[Proof of Theorem \ref{mainthm} given Theorem \ref{reducedthm}]
Let $d_{\Phi}$ be a Carnot-Carath\'eodory metric on $G_{\infty}$ and suppose that the associated norm $\Phi$ on $\g^{ab}$
is conjugation-invariant. Given any Cayley graph of $\Gamma$, use Theorem \ref{reducedthm} to choose ergodic stationary weights $w$ 
such that the resulting $T$ is bi-Lipschitz to $d$ and such that the norm on $\g^{ab}$ associated to the metric $\E T$ on $\Gamma$
is equal to $\Phi$.
Applying Theorem \ref{cantrellfurmanthm} to $w$ then gives
\[
   \left(\Gamma, \frac{1}{n} T\right) \tendsto{n}{\infty} (G_{\infty}, d_{\Phi})
\]
in the sense of pointed Gromov-Hausdorff convergence, as desired.
\end{proof}

Thus, our main theorem is reduced to the problem of constructing stationary weights which induce a given norm $\Psi$ on $\g^{ab}$.
Haggstrom and Meester ~\cite{HaggstromMeester} give a construction for inducing the correct norms in the $\Z^d$ case, and
in the simplest case, the core
of our work is 
``lifting'' the Haggstrom-Meester construction from the abelianization of the finitely generated nilpotent group to the group itself, and then 
checking that
everything goes through.
Therefore, to give an idea of the construction
we start by proving Theorem \ref{reducedthm} in this simplest case---namely, the case that $\Gamma=N$ is a torsion-free nilpotent group
with torsion-free abelianization, and the generating set $S$ projects to the standard generating set of
$\Z^d \cong N^{ab} = \Gamma^{ab}$. As mentioned above, in this case conjugation-invariance does not play a role, and \emph{any}
norm $\Psi$ is attainable.
This is done in the next two sections.  

In Section \ref{restrictions}, we discuss the restriction of conjugation-invariance and the
nontrivial subtleties that arise when treating the general \emph{virtually} nilpotent case.
The rest of the main body of the paper is then dedicated to proving Theorem \ref{reducedthm} in full generality.
In particular, this involves understanding a \emph{virtually} abelian ``almost-abelianization'' of $\Gamma$, and
then again ``lifting'' a construction from the ``almost-abelianization'' to $\Gamma$.
In order to accommodate all possible Cayley graphs
as well as the slightly non-abelian nature of the ``almost-abelianization'',
the general construction has a ``coarser'' flavor than the original construction and requires some non-trivial modifications.

Appendix \ref{liegroupconstructions} provides more background on the associated graded nilpotent Lie group and Carnot-Carath\'eodory metrics.
Appendix \ref{grouptheory} shows that the construction at the end of Section \ref{defnsandbackground} coincides with the construction
in Cantrell-Furman's theorem ~\cite{CantrellFurman}.
In Appendix \ref{fillgap}, we review the notion of Gromov-Hausdorff convergence and we also provide a missing step in Cantrell-Furman's theorem so that it guarantees Gromov-Hausdorff convergence.

 \section{Construction of the edge weights when $\Gamma$ is nilpotent and torsion-free with torsion free abelianization} \label{easyconstruction}
 Assume that $\Gamma = N$ is a finitely generated torsion-free nilpotent group with torsion-free abelianization.
 Moreover, assume that $S=\{s_1,...,s_d\}$ is such that the image of $S$ under the quotient map $\Gamma \to \Gamma^{ab}$
 is a basis, and we choose an isomorphism $\Gamma^{ab} \cong \Z^d$ such that $S$ maps to the standard basis for $\Z^d$. In this 
 and the next section we prove the result of Theorem \ref{reducedthm}\footnote{Technically we prove a weaker version of Theorem
 \ref{reducedthm} which still implies the conclusion of Theorem \ref{mainthm}; see Remark \ref{bilipschitzcaveat} below.}
  under these extra assumptions, which then implies the result
 of Theorem \ref{mainthm} under these extra assumptions, as shown above.

 First, let us note that since $\Gamma$ is nilpotent, we cannot have $d=0$,
 and if $d=1$ then in fact $\Gamma \cong \Z$. (For this latter fact, let $a \in \Gamma$ be such that 
 $\langle a \rangle [\Gamma, \Gamma] = \Gamma$; then also $\langle a \rangle = \Gamma$ by Theorem 16.2.5
 in ~\cite{KarMer}). It is easy to induce any norm on $\Z$ no matter what the finite generating set is using deterministic weights,
 so from here on we assume $d \ge 2$. 
 
 We are given a norm $\Phi$ on $\Gamma^{ab} \otimes \R \cong \R^d$. We want to find weights $w:E \to \R_{\ge 0}$ for $\Gamma$
 such that the subadditive function $\tilde{T}$ on $\Gamma^{ab} \cong \Z^d$ induced by $\E T$ via $\Gamma \to \Gamma^{ab}$ is
 asymptotically equivalent to $\Phi$. Let $B \subset \R^d \cong \Gamma^{ab} \otimes \R$ be the unit ball of $\Phi$. 
 Note that $B$ is a compact, convex, and symmetric (i.e. $x \in B$ implies $-x \in B$) subset
 of $\R^d$ which contains an open neighborhood of $0$. The construction below is a ``lift'' of the construction of 
 Haggstrom and Meester ~\cite{HaggstromMeester}.
 
 Let $\{b_n\}_{n=1}^{\infty}$ be a countable dense subset of the boundary of $B \subset \R^d$. For each $n \ge 1$,
 let $z_n$ be a point in $\Z^d$ with minimum possible distance to $2^n \frac{b_n}{\| b_n \|_2} \in \R^d$, where $\| \cdot \|_2$
 is the standard Euclidean norm on $\R^d$. We recall the result from \cite{HaggstromMeester} that we need in the proposition below.
 \begin{prop} \label{pathsprop}
 There is a constant $C_0$ depending only on $d$ such that, for any $n \ge 1$, $u \in \R^d$, if $z_n$ is a point in $\Z^d$ with minimal
 Euclidean to $2^n u$,
 there exists a directed edge path $\gamma_n$ from $0$ to $z_n$ in the standard Cayley graph $\Z^d$ with the following properties:
 \begin{enumerate}
\item Any point on $\gamma_n$ is Euclidean distance at most $C_0$ from some point on the line through $0$ and $b_n$ in $\R^d$
\item If a subpath of $\gamma_n$ starts at $x \in \R^d$ and ends at $y \in \R^d$, then $\pair{y-x}{b_n} > 0$.
\item The number of edges in $\gamma_n$ is the least possible, i.e. $\sum_{i=1}^d |\pi_i (z_n) |$, where $\pi_i: \R^d \to \R$
is projection onto the $i^{th}$ coordinate.
\end{enumerate}
 \end{prop}

Next, we lift each $\gamma_n$ to an edge path $\bar{\gamma}_n$ in the Cayley graph of $\Gamma$. The
quotient map $\Gamma \to \Gamma^{ab} \cong \Z^d$ induces a covering map of Cayley graphs, 
so just let $\bar{\gamma}_n$ be the unique lift of $\gamma_n$ starting at
$1 \in \Gamma$. Equivalently, paths in Cayley graphs starting at the identity are naturally in correspondence with words in the
generating sets. The path $\gamma_n$ then corresponds to a word in $e_1,...,e_d$, which we lift to a word in $s_1,...,s_d$,
which corresponds to a path $\bar{\gamma}_n$ in our Cayley graph for $\Gamma$.

For each $n \ge 1$, set $E_n \subset E$ to be the set of edges of the Cayley graph of $\Gamma$ which share at least one
vertex in common with an edge of $\bar{\gamma}_n$. Note that $|E_n| \lesssim 2^n$, where the implied constant depends on $|S|$
but is independent of $n$.

Now we define a configuration of edge weights $\eta_n : E_n \to \R_+$.
First choose $h>0$ sufficiently small so that $\{ x \in \R^d : \|x\|_2 \le h \} \subset B$.
Next, choose $K< \infty$ sufficiently large so that $ \frac{ 1 }{ K - 2 h^{-1} \cdot C_0 } \le h$ and $K \ge h^{-1}$.
We then define
\begin{equation*}
  \eta_n(f) = \begin{cases}
                        \frac{|\pi_i(b_n)|}{\|b_n\|_2^2}   & f \in \bar{\gamma}_n, f \mbox{ labeled by } s_i,\\
                        K,                                              & \text{otherwise}
                     \end{cases}
\end{equation*}
where $\pi_i$ is again the projection onto the $i^{th}$ coordinate. If $x \in \Gamma$, then we can also define the translated configuration $T_x \eta_n : xE_n \to \R_+$
by $T_x \eta_n(f) = \eta_n(x^{-1} f)$

Let $(Y_x)_{x \in \Gamma}$ and $(Z_x)_{x \in \Gamma}$ be collections of i.i.d. random variables with distributions that satisfy 
$\Prob(Y_x = 0) = \frac{1}{2}, \Prob(Y_x = n) = 3^{-n}$ for $n\geq 1$, and
$Z_x$ is uniformly distributed on $[0,1]$. We also assume that the collections $(Y_x)_{x \in \Gamma}$, $(Z_x)_{x\in \Gamma}$ are independent.

Finally, the weights $w:E \to \R_+$ are defined as follows: if $Y_x = n > 0$, assign the edges in $x E_n$ according to $T_x \eta_n$.
If two configurations compete for the same edge, then the configuration with the larger value of $n$ wins; if both configurations
have the same value of $n$, then the one with the larger value of $Z_x$ wins. Any remaining edges with no assigned weight are
given weight $K$.

More formally: for each $f \in E$, let $X_f := \{ x \in \Gamma : f \in xE_{Y_x} \}$ be the set of starting points of configurations competing for
the edge $f$. Let $n_f := \max \{Y_x : x \in X_f \}$ be the largest value of $n$ among these competing configurations, 
and let $x_f \in \Gamma$ be the element of $X_f$ which attains the maximum (that is, $Y_{x_f} = n_f$) and has the largest
value of $Z_x$ among such elements, that is, $Z_{x_f} = \max \{Z_x : x \in X_f, Y_x = n_f\}$. Then
\begin{equation*}
  w(f) = \begin{cases}
              T_{x_f} \eta_{n_f} (f) & X_f \ne \emptyset \\
              K                               & \text{otherwise}.
            \end{cases}
\end{equation*}

Note that $x_f$ is a.s. unique since all the $Z_x$ are uniform, and it exists since $|X_f| < \infty$ a.s. by the calculation
\[
\E |X_f| = \sum_{x \in \Gamma} \Prob(f \in xE_{Y_x}) 
= \sum_{n=1}^{\infty} \sum_{x \in \Gamma} \ind_{\{f \in xE_n\}} \Prob(Y_x = n)
\le \sum_{n=1}^{\infty} |E_n| 3^{-n} \lesssim \left( \sum_{n=1}^{\infty} 2^n \cdot 3^{-n} \right) < \infty.
\]
Here we used that $\Gamma$ acts freely on $E$ and so $\#\{x \in \Gamma : x^{-1} f \in E_n\} \le |E_n|$.
Hence the weights are well-defined. They are also evidently stationary and a.s. bounded above by $K < \infty$.
The weights are also ergodic, since we can take our probability space $\Omega$ to be $(\mathbb{N} \times [0,1])^{\Gamma}$,
corresponding to the outcomes of $Y_x$ and $Z_x$, which is clearly ergodic as a direct product of probability spaces over $\Gamma$.
\begin{rmk} \label{bilipschitzcaveat}
These weights do \emph{not} give a metric which is bi-Lipschitz to a word metric, since $\pi_i(b_n)$ will typically cluster around $0$ and a uniform lower bound on the edge weights is not available. 
\end{rmk}

By the remark above, this construction does not suffice to prove Theorem \ref{reducedthm}. There are two ways around this.  
In Section \ref{generalconstruction}, we provide a different construction in the general virtually nilpotent case which is bi-Lipschitz to the word metric, and implies Theorem \ref{reducedthm} as stated.
Secondly, the weights constructed above \emph{do} satisfy a weaker condition which one might call ``bi-Lipschitz away from the diagonal.''
That is, we have a uniform upper bound $K$ on the edge weights, 
and there exists some constants $0<C<\infty$ and $k>0$ such that
for any $x,y \in \Gamma$ with $d(x,y) \ge C$, we have \[T(x,y) \ge k d(x,y)\] almost surely.
This fact follows fairly easily from Lemma \ref{euclideanmonotonicity} proven in Section \ref{coarsepathssection} below.

Under this weaker assumption, the proof of Theorem \ref{cantrellfurmanthm} given in ~\cite{CantrellFurman} goes through unchanged.
Thus, although we prove a weaker version of Theorem \ref{reducedthm}
in the next section, namely Theorem \ref{reducedthm} with the conclusion ``$T$ is bi-Lipschtiz to $d$'' replaced by 
the conclusion ``$T$ is bi-Lipschitz to $d$ away from the diagonal'', we can then use the stronger version of Theorem \ref{cantrellfurmanthm}
to still conclude the result of Theorem \ref{mainthm} in this restricted setting.

\section{Proof of Theorem \ref{reducedthm} when $\Gamma$ is nilpotent and torsion-free with torsion free abelianization}
Using the weights $w$ defined in the previous section, let $T$ be the metric associated to $w$ as defined in Section \ref{defnsandbackground}.
Let $\tilde{T}$ be the subadditive function on $\Gamma^{ab}$ induced by $\E T$ via the abelianization map $\Gamma \to \Gamma^{ab}$
as above. In order to prove our version of Theorem \ref{reducedthm}, all that remains is to show that
as $x \in \Gamma^{ab}$ tends to infinity,
\[ \tilde{T}(x) - \Phi (x) = o(x), \]
where in the little $o$ notation we may use any norm on $\R^d$ to measure $x$.
We use the following proposition which is used in ~\cite{HaggstromMeester} (where they take $Q=[-1/2,+1/2]^d \subset \R^d$,
but the exact form that $Q$ takes does not matter):
\begin{prop} \label{normreductionprop}
To show that $\tilde{T}(x) - \Phi(x) = o(x)$, it suffices to show the following
\begin{enumerate}
\item For all $y \notin B$, $y \notin \frac{1}{t}\bar{B}(t)$ for all sufficiently large $t$. 
\item For all $y$ in the interior of $B$, $y \in \frac{1}{t}\bar{B}(t)$ for all sufficiently large $t$. 
\end{enumerate}
Here we define 
\[ \bar{B}(t) := \bigcup_{\{x \in \Gamma^{ab} : \tilde{T}(x) \le t \}} x + Q, \]
where $Q \subset \g^{ab}$ is a compact connected neighborhood of $0$ such that the quotient map $Q \to \g^{ab}/\Gamma^{ab}$
is surjective.
\end{prop}

First, we prove (1). 
To do this, we must establish some facts about the relationship
between the $T$-lengths of paths in $E$ and their ``displacements'' in $\Gamma^{ab}$.
In proving these we will repeatedly use the following easily
verifiable lemma from ~\cite{HaggstromMeester}:
\begin{lemma} \label{convexlemma}
   Let $B$ be a convex subset of $\R^d$ and let $x_1,...,x_m \in \R^d$, $\alpha_1,...,\alpha_m \ge 0$ be such that
   each $\alpha_i^{-1} x_i \in B$. Then $\frac{x_1 + \cdots + x_n}{\alpha_1 + \cdots + \alpha_n} \in B$.
\end{lemma}

Let us call an edge $f \in E$ ``slow'' if $w(f) = K$ and ``fast'' otherwise. Let us also call an edge path in $E$ ``fast'' if all its edges are fast
and ``slow'' if all its edges are slow. For an edge path $\gamma$ in $E$ from $x \in \Gamma$ to $y \in \Gamma$ denote by $D(\gamma)$
its ``displacement'' $y^{ab} - x^{ab} \in \R^d$. Note that displacement is preserved by left translations:
\[ D(z\gamma) = (zy)^{ab} - (zx)^{ab} = (z^{ab} + y^{ab}) - (z^{ab} + x^{ab}) = y^{ab} - x^{ab} = D(\gamma). \] 

Let us first consider fast paths $\gamma$. Note that by construction of the weights, each fast path is a subpath of $x\bar{\gamma}_n$ for some
$x \in \Gamma, n \ge 1$ (because of the ``shell'' of slow edges surrounding each fast $x \bar{\gamma_n}$). 
We can then decompose $D(\gamma)$ as
\[ 
   D(\gamma) = D_{\parallel}(\gamma) + D_{\perp}(\gamma), 
\]
where $D_{\parallel}$ is the orthogonal projection of $D(\gamma)$ onto the line passing through $0$ and $b_n$ and $D_{\perp}(\gamma)$
is orthogonal to that line. Note that the construction of the edge weights guarantees precisely that if $f$ is a fast edge in 
$x\bar{\gamma}_n$ labeled by $s_i$ then
\[ \frac{D_{\parallel}(f)}{T(f)} = \frac{ \pair{ \pm e_i }{ \frac{b_n}{\|b_n\|_2} } \frac{b_n}{\| b_n \|_2 } }{ \frac{ |\pi_i(b_n)| }{ \|b_n\|_2^2 }}
= \pm b_n \in B. \]
Then by Lemma \ref{convexlemma} we have
\[ \frac{ D_{\parallel}(\gamma) }{ T(\gamma) } = \frac{ \sum_{f \in \gamma} D_{\parallel}(f) }{ \sum_{f \in \gamma} T(f) } \in B. \]
We also know by Proposition \ref{pathsprop} that 
\[ 
   \|D_{\perp}(\gamma)\|_2 \le 2 C_0, 
\]
and hence
\[ 
   \frac{ D_{\perp}(\gamma) }{ h^{-1} \cdot 2 C_0 } \in \{ x \in \R^d : \|x\|_2 \le h \} \subset B. 
\]
So again by Lemma \ref{convexlemma},
\[ 
   \frac{ D(\gamma) }{ T(\gamma) + 2 h^{-1} C_0 } = \frac{ D_{\parallel}(\gamma) + D_{\perp}(\gamma) }{ T(\gamma) + h^{-1} \cdot 2C_0 } \in B.
\]

On the other hand, if $f$ is a slow edge, then by our choice of $K$
\[ 
   \frac{ D(f) }{ T(f) - 2 h^{-1} C_0 } \in \{ x \in \R^d : \|x\|_2 \le h \} \subset B,
\]
and so for a slow path $\gamma$, by Lemma \ref{convexlemma} we have
\[
   \frac{ D(\gamma) }{ T(\gamma) - 2 |\gamma| h^{-1} C_0} \in B.
\]

Now, a general path in $E$ is an alternating concatenation of fast and slow paths. That is,
$\gamma = \gamma_f^0 \gamma_s^1 \cdots \gamma_s^n \gamma_f^n$, where the $\gamma_f^i$ are fast, the $\gamma_s^i$ are slow,
and we may take $\gamma_f^0$ or $\gamma_f^n$ to be empty, but all the $\gamma_s^i$ consist of at least one edge. Then by 
our previous arguments and Lemma \ref{convexlemma} we have
\[
   \frac{\sum_{i=0}^n D(\gamma_i^f) + \sum_{i=1}^n D(\gamma_i^s)}
   {\sum_{i=0}^n (T(\gamma_i^f) + 2h^{-1} C_0 ) + \sum_{i=1}^n (T(\gamma_i^s) - 2 |\gamma_i^s| h^{-1} C_0)} \in B.
\]
The numerator in the above expression is $D(\gamma)$, and the denominator is at most $T(\gamma) + 2 h^{-1} C_0$,
so we have
\[
   \frac{D(\gamma)}{T(\gamma) + 2h^{-1} C_0} \in B
\]
for any path $\gamma$ in $E$.

Finally, let $y \notin B$. Since $B$ is closed, there is some $\epsilon > 0$ such that for any $c>0$, $cB(y,\epsilon) \cap B \ne \emptyset$
implies that $\frac{1}{c} > 1 + \epsilon$.
Now for any $t > 0$ let $z \in \Gamma$ be such that $ty \in z^{ab} + Q$, where $Q$ is the fixed compact set in 
Proposition \ref{normreductionprop}. If we choose $\gamma$ to be a $T$-minimal path from $1$ to $z$ in $\Gamma$, by our above arguments
we have that
\[
   \frac{ z^{ab} }{ T(\gamma) + 2 h^{-1} C_0 } = 
   \frac{  t[ y - \frac{1}{t}(z^{ab} - ty) ] }{ T(1,z) + 2 h^{-1} C_0 } \in B.
\]
Therefore, whenever $\frac{ \diam(Q) }{ t } < \epsilon$, we have 
$\frac{1}{t}\|z^{ab} - ty\|_2 < \epsilon$ and hence
\[ \frac{ T(1,z) + 2h^{-1} C_0}{t} > 1 + \epsilon; \]
and so whenever also $\frac{2 h^{-1} C_0 }{t} < \epsilon/2$, we have
\[ \frac{T(1,z)}{t} > 1 + \frac{\epsilon}{2}, \]
and then taking expectation gives
\[ \frac{ \E T(1,z) }{t} > 1 + \frac{\epsilon}{2}; \]
since this argument did not depend on our choice of $z$, we conclude that, for all $t$ sufficiently large,
$\tilde{T}(z^{ab}) > t(1 + \frac{\epsilon}{2})$ whenever $ty \in z^{ab} + Q$, and hence 
\[ y \notin \frac{\bar{B}(t)}{t}. \]

Now we prove (2).

It is sufficient to prove that for every $\epsilon > 0$, for all but finitely many $n$,
\[ \frac{ \|b_n\|_2 \tilde{T}(z_n) }{ 2^n } < 1 + \epsilon. \]

Fix $\epsilon > 0$. We give an upper bound on the $\tilde{T}$-distance from $0$ to $z_n$ 
by constructing a path $\gamma$ from $1$ to a lift of $z_n$ in $\Gamma$. The lift we choose is the endpoint of the path $\bar{\gamma}_n$,
which we denote by $\bar{z}_n$. Note that although the path we construct is random, the endpoints $1$ and $\bar{z}_n$ are not.

Denote by $Z$ the center of $\Gamma$, and fix a total ordering $<$ on $Z$ such that if $d(1,x_0) < d(1,x_1)$, then $x_0 < x_1$
(recall that here $d$ denotes the word metric on $\Gamma$ with respect to $S$).
Then choose $x$ to be the least element of $Z$ with respect to this ordering such that $Y_x = n$.
Note that $x$ is then a well-defined $Z$-valued random variable with minimal distance from $1$, and that
\[
   (x=x_0) \Leftrightarrow (Y_{x_0} = n \mbox{ and } Y_{x_1} \ne n \mbox{ for all } x_1 < x_0).
\]
That is, $x$ is the nearest central starting point of a ``highway'' in the $b_n$ direction.

Now, to construct our path $\gamma$, first, take a path of minimal $d$-length from $1$ to $x$ in $\Gamma$.
Then, travel along $x \bar{\gamma}_n$ (even if some of the edges are overwritten by slow edges) to $x \bar{z}_n$. Finally, travel back to
$x \bar{z}_n x^{-1} = \bar{z}_n$ by traveling backwards along a translate of the path you took from $1$ to $x$. Note that we have used
the fact that $x$ is central to conclude that $x \bar{z}_n x^{-1} = \bar{z}_n$ and in particular that the $d$-distance from $x \bar{z}_n$
to $\bar{z}_n$ is no larger than the $d$-distance from $1$ to $x$.

If $x \bar{\gamma}_n$ was not overwritten by any slow edges, the passage time of the path would be equal to
\[ 
   \sum_{f \in \bar{\gamma}_n} \eta_n(f) = \sum_{f \in \bar{\gamma}_n} \frac{ \pair{D(f)}{ b_n } }{ \|b_n\|_2^2 } 
   = \frac{ \pair{D(\bar{\gamma}_n)}{b_n} }{\|b_n\|_2^2} = \frac{ \pair{z_n}{b_n} }{\|b_n\|_2^2}.
\]
(Here we have used the fact that, by construction, all edges $f$ in $\gamma$ have positive inner product with $b_n$.)
Since $z_n$ is less than distance $\frac{\sqrt{d}}{2}$ from $2^n \frac{b_n}{\|b_n\|}$, the above is bounded above by
\[ 
   \frac{ \pair{2^n \frac{b_n}{\|b_n\|} }{b_n} }{ \|b_n\|_2^2 } + \frac{ \frac{\sqrt{d}}{2} \cdot \| b_n\|_2 }{\|b_n\|_2^2}
   = \frac{ 2^n }{ \|b_n\| } \left(1 + \frac{ \sqrt{d} }{ 2^{n+1} } \right). 
\]
Taking into account the travel from $1$ to $x$ and from $x \bar{z}_n$ to $\bar{z}_n$, as well as the fact that some of the edges
of $x \bar{\gamma}_n$ may be overwritten by slow edges, we have
\begin{equation}
\label{eqn:pathlengthbound}
   \E T(\gamma) \le K \big[ 2 \E d(1,x) + \E \# \{ e \in x \bar{\gamma}_n : e \mbox{ is slow} \} \big]
   + \frac{ 2^n }{ \|b_n\| } \left(1 + \frac{ \sqrt{d} }{ 2^{n+1} } \right).
\end{equation}

To bound the first term, we calculate
\[ 
   \E d(1,x) = \sum_{i=0}^{\infty} \Prob( d(1,x) > i ) = \sum_{i=0}^{\infty} \Prob( Y_{\xi} \ne n \mbox{ for all } \xi \in B_{d}(i) \cap Z ). 
\]
Since we have assumed that $\Gamma \not\cong \Z$, the growth of the center is at least
2-dimensional, that is, we have some
$C > 0$ depending only on $\Gamma$ and $S$ such that
\[ 
   |B_{d}(i) \cap Z| \ge Ci^2 
\]
for all $i \ge 0$. This is proved in Lemma \ref{centergrowth} below, but for now we take it for granted.

Then, since the $Y_\xi$ are iid, we continue the above computation to get
\[ 
   \E d_S(1,x) \le \sum_{i=0}^{\infty} (1 - 3^{-n})^{C i^2} \le 1 + \int_0^{\infty} (1 - 3^{-n})^{C s^2} ds. 
\]
Using the substitution $\sigma =  \left[ \frac{ \ln (1 - 3^{-n}) }{\ln (1 - 3^{-1}) } \right]^{1/2} s$, we get
\[ 
   \int_{0}^{\infty} (1 - 3^{-n})^{Cs^2} ds =  
   \left[ \frac{ \ln (1 - 3^{-n}) }{\ln (1 - 3^{-1}) } \right]^{-1/2} \int_0^{\infty} (1 - 3^{-1})^{C \sigma^2} d\sigma,
\]
which is to say that 
\[ 
   \E d(1,x) \le 1 + C' [- \ln (1 - 3^{-n} ) ]^{-1/2} 
\]
for some $C' > 0$ independent of $n$.
By convexity, $- \ln ( 1 - s ) \ge s $ for all $s < 1$, and so 
\[ [ - \ln(1 - 3^{-n}) ]^{-1/2} \le (3^{-n})^{-1/2} = 3^{n/2}, \]
thus 
\begin{equation}
   \label{eqn:suburbbound}
   \E d(1,x) \lesssim 3^{n/2},
\end{equation} 
the implied constant of course independent of $n$.

Now, we bound 
\[ \E \#\{e \in x \bar{\gamma}_n : e \mbox{ is slow }\} = \sum_{e \in \bar{\gamma}_n } \Prob( xe \mbox{ is slow} ); \]
since $xe$ will only be slow if another $T_z E_{Y_z}$ with $Y_z \ge n$ competes for it, the above quantity is bounded above by
\begin{align*}
   &\sum_{e \in\bar{\gamma}_n} \Prob(xe \in zE_{Y_z} \mbox{ and } Y_z \ge n \mbox{ for some } x \ne z \in \Gamma) \\
   &\le \sum_{e \in \bar{\gamma}_n}  \sum_{x_0 \in \Gamma} \sum_{z \in \Gamma \setminus x_0} \sum_{i=n}^{\infty}
   \Prob(x=x_0, x_0 e \in z E_i, Y_z = i) \\
   & = \sum_{e \in \bar{\gamma}_n} \sum_{x_0 \in \Gamma} \sum_{i=n}^{\infty} \sum_{z \in \Gamma : x_0^{-1} z e \in E_i} 
   \Prob(x=x_0, Y_z = i); 
\end{align*}
we claim that for $i \ge n$ and $x_0 \ne z$, $\Prob(x=x_0, Y_z=i) \le \frac{3}{2} \Prob(x=x_0) \Prob(Y_z=i)$, and hence we continue
\begin{align}
   \label{eqn:obstructionbound} 
   \E \#\{e \in x \bar{\gamma}_n : e \mbox{ is slow }\} 
   &\le \sum_{e \in \bar{\gamma}_n} \sum_{x_0 \in \Gamma} \sum_{i=n}^{\infty} \sum_{z \in \Gamma : x_0^{-1} z e \in E_i} 
   \frac{3}{2} \Prob(x=x_0) \Prob(Y_z = i) \nonumber \\
   &\le \frac{3}{2}\sum_{e \in \bar{\gamma}_n} \sum_{x_0 \in \Gamma} \sum_{i=n}^{\infty} |E_i| \Prob(x=x_0) \Prob(Y_z=i) 
   = \frac{3}{2} \sum_{e \in \bar{\gamma}_n} \sum_{i=n}^{\infty} |E_i| \Prob(Y_z=i) \nonumber \\
   &\lesssim \sum_{e \in \bar{\gamma}_n} \sum_{i=n}^{\infty} 2^i \cdot 3^{-i}
   = \sum_{e \in \bar{\gamma}_n} 3 \left( \frac{2}{3} \right)^n 
   = 3 |\bar{\gamma}_n| \left( \frac{2}{3} \right)^n \nonumber \\
   &\lesssim \left( \frac{4}{3} \right)^n. 
\end{align}
To prove the claim, note that for $x_0 \ne z$, $i \ge n$,
\[ \Prob(x = x_0, Y_z = i) = \Prob(Y_{x_1} \ne n \mbox{ for all } x_1 < x_0, Y_{x_0} = n, Y_z = i); \]
if $x_0 < z$, then all these events are independent, and hence $\Prob(x=x_0, Y_z=i) = \Prob(x=x_0)\Prob(Y_z=i)$.
Otherwise $z < x_0$, and then
\[ \Prob(x = x_0, Y_z = i) = \left(\prod_{x_1 < x_0, x_1 \ne z} \Prob(Y_{x_1} \ne n) \right) \Prob(Y_{x_0} = n) \Prob( Y_z \ne n, Y_z = i). \]
If $i = n$, then this is equal to $0$. Otherwise, $i>n$, and 
\[
    \Prob(Y_z \ne n, Y_z = i) = \Prob(Y_z = i) = \frac{\Prob(Y_z = i)}{\Prob(Y_z \ne n)} \Prob(Y_z \ne n) \le \frac{3}{2} \Prob(Y_z = i) \Prob(Y_z \ne n),
\]
where we used that $\Prob(Y_z \ne n) = 1 - 3^{-n} \ge \frac{2}{3}$. Hence
\begin{align*}
   \Prob(x=x_0, Y_z = i) 
   &\le \frac{3}{2} \left(\prod_{x_1 < x_0, x_1 \ne z} \Prob(Y_{x_1} \ne n) \right) \Prob(Y_{x_0} = n) \Prob(Y_z \ne n) \Prob(Y_z = i) \\
   &= \frac{3}{2} \Prob(x=x_0) \Prob(Y_z = i),
\end{align*}
as desired.

Hence, applying \eqref{eqn:pathlengthbound}, \eqref{eqn:suburbbound}, and \eqref{eqn:obstructionbound},
\[ 
   \frac{ \|b_n\|_2 \tilde{T}(z_n)}{2^n} \le \frac{ \| b_n \|_2 \E T(\gamma) }{2^n}
   \le   K \|b_n\|_2 \left[ 2 O \left(\left(\frac{3^{1/2}}{2}\right)^n\right) + O\left(\left(\frac{2}{3}\right)^n\right) \right] + 1 + \frac{\sqrt{k}}{2^{n+1}}, 
\]
which is less than $1 + \epsilon$ for sufficiently large $n$, as desired.

To tie up the final loose end, we prove that the volume growth of the center of $\Gamma$ is at least $2$-dimensional.
This is a simple corollary of the following lemma from the notes of Drutu and Kapovich ~\cite{GGTGrutuKapovich}:
\begin{lemma}[Lemma 14.15 from \cite{GGTGrutuKapovich}] \label{ggtlemma}
   Let $\Gamma$ be a finitely generated nilpotent group of class $k$ and let $C^k \Gamma$ be the last nontrivial term in its lower central series.
   If $S$ is a generating set for $\Gamma$, and $g \in C^k \Gamma$, then there exists a constant $\lambda = \lambda(S,g)$ such that
   for all $m \ge 0$,
   \[
      d_S (1, g^m) \le \lambda m^{1/k}.
   \]
\end{lemma}
\begin{lemma} \label{centergrowth}
   Let $\Gamma$ be a nontrivial finitely generated torsion-free nilpotent group which is not isomorphic to $\Z$, 
   $S$ a finite generating set for $\Gamma$.
   Denote the center of $\Gamma$ by $Z$.
   Then, there exists a constant  $C > 0$ depending only on $\Gamma$ and $S$ such that
   \[
      \# \{z \in Z : d(1,z) \le i\} \ge Ci^2
   \]
   for all $i \ge 0$.
\end{lemma}
\begin{proof}
   We know that $Z$ is a nontrivial finitely generated free abelian group. First, assume that $Z \not\cong \Z$. Then $Z \cong \Z^k$ for some
   $k \ge 2$. Then the lemma follows, since the quantity in question grows at least as fast as $Z$ does as a finitely generated group.
   More explicitly, if $S'$ is a finite generating set for $Z \cong \Z^k$, we know that there exists $C'>0$ depending only on $S'$ such that
   \[
      \# \{z \in Z : d_{S'}(1,z) \le i \} \ge C' i^k.
   \]
   Take $m = \max_{s \in S'} d (1,s) < \infty$. Then for all $z \in Z$, $d(1,z) \le m d_{S'}(1,z)$, and hence
   \[
      \# \{z \in Z : d(1,z) \le i\} \ge \# \left\{z \in Z : d_{S'}(1,z) \le \frac{i}{m} \right\} \ge \frac{C'}{m^k} i^k.
   \]

   Now, suppose $Z \cong \Z$. Then $\Gamma$ is not abelian (otherwise we would have $\Gamma = Z \cong \Z$, contradicting our assumption).
   So $\Gamma$ is nilpotent of step $k$ for some $k \ge 2$, and $C^k \Gamma$ is a nontrivial subgroup of $Z$.
   Take a generator $g$ for $C^k \Gamma$. By Lemma \ref{ggtlemma}, we get
   $\lambda=\lambda(g,S)>0$ such that $d(1, g^m) \le \lambda m^{1/k}$ for all $m \ge 0$. Therefore
   \[
      \{z \in Z : d(1,z) \le i\} \ge \{m \ge 0 : d(1, g^m) \le i\} \ge \{m \ge 0 : \lambda m ^{1/k} \le i\} 
      \ge \left\lfloor \frac{1}{\lambda^k} i^k \right\rfloor \ge C i^k
    \]
    for some $C >0$.
\end{proof}

\section{Restrictions in the virtually nilpotent case} \label{restrictions}

Any finitely generated virtually nilpotent group $\Gamma$ will contain a finite index subgroup $H$ which is finitely generated, nilpotent, torsion free, 
and which has torsion-free abelianization (see Appendix \ref{grouptheory}). We often think of the $H$ and $\Gamma$ as having the same coarse
geometry; indeed:
\begin{prop} \label{prop:finiteindex}
Let $\Gamma$ be a group endowed with a metric $T$, let $H$ be a finite index subgroup, and let $(X,D)$ be a metric space. If $T \lesssim d$ ($d$
the word metric) and $(H, \frac{1}{t} (T|_H)) \xrightarrow{GH} (X,D)$, then also $(\Gamma, \frac{1}{t} T) \xrightarrow{GH} (X,D)$.
\end{prop}
\begin{proof}
Since $(H, \frac{1}{t} T|_H)$ is a metric subspace of $(\Gamma,  \frac{1}{t} T)$, 
the Gromov-Hausdorff distance between the two spaces is bounded---up to an absolute constant---by
\[ \inf \{ \epsilon > 0 : T(g,H) < \epsilon \mbox{ for all } g \in \Gamma \}, \]
which is itself bounded up to a constant by 
\[ \frac{1}{t}[\Gamma:H] = O(1/t). \]
Thus $(\Gamma,\frac{1}{t}T)$ and $(H, \frac{1}{t}T)$ must tend to the same limit.
\end{proof}

Thus, it might seem trivial to pass from the simplified case we just proved to the general case.
However, perhaps surprisingly, the answer to the question we consider is \emph{not} the same for $\Gamma$ and $H$. In general, there may be some limit shapes for stationary FPPs on $H$ which are \emph{not} attained by stationary FPPs on $\Gamma$. Consider the following example.

\begin{figure}[t]
   \centering
   \includegraphics[scale=1]{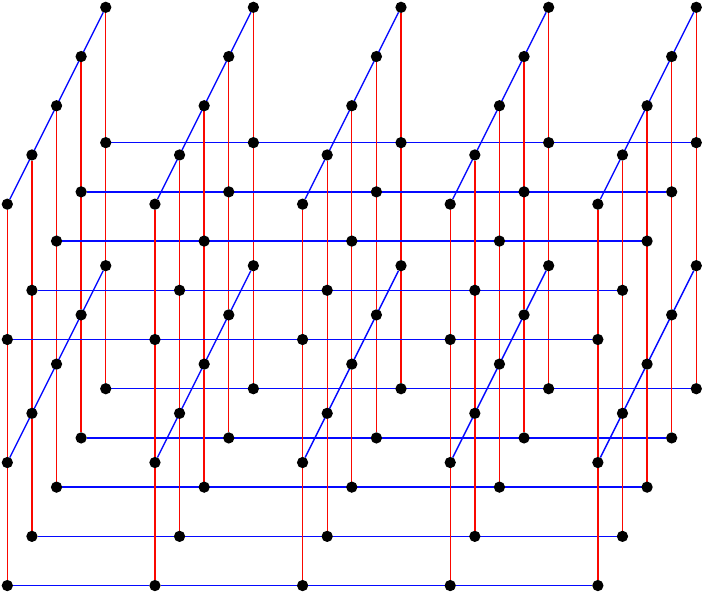}
   \caption{A portion of the Cayley graph of $\langle \rho \rangle \ltimes \mathbb{Z}[i]$ 
   with respect to the generating set $\{\rho, 1+0i\}$. 
   Edges labeled by $\rho$ are red, while edges labeled by $1+0i$ are blue.}
\end{figure}

Let $\Gamma := \langle \rho \rangle \ltimes \mathbb{Z}[i]$, the semidirect product of the Gaussian integers with a cyclic group of order four, the 
generator of the cyclic group acting by multiplication by $i$. $\Gamma$ contains the abelian (hence nilpotent) group $\Z[i] \cong \Z^2 =: H$ as
a subgroup of index $4$. We know from our work above (and from ~\cite{HaggstromMeester}) that any norm on $\R^2$ is attainable as a limit shape 
for $H$. However, we claim that the scaling limit of any invariant metric on $\Gamma$ which is $\lesssim d$ (such as $\E T$ for a stationary FPP 
$T$ with integrable weights) must be a norm on $\R^2$
which has $\frac{\pi}{4}$ rotational symmetry. Take any $(x + iy) \in \Z[i]$. Then
\begin{align*} 
   \E T(1,i(x+iy)) &= \E T(1, \rho^{-1}(x+iy)\rho) \\
   &\le \E T(1, \rho^{-1}) + \E T(\rho^{-1}, \rho^{-1}(x+iy)) + \E T(\rho^{-1}(x+iy), \rho^{-1}(x+iy)\rho) \\
   &= \E T(1, \rho^{-1}) + \E T(1, (x+iy)) + \E T(1, \rho) \le \E T(1, (x+iy)) + 2(\mathrm{const.}).
\end{align*}
Iterating this inequality four times and taking a scaling limit gives 
\begin{align*}
   \lim_{n \to \infty} \frac{\E T(1, n(x+iy))}{n}& \\
   =  \lim_{n \to \infty} \frac{\E T(1, ni(x+iy))}{n}
   &=\lim_{n \to \infty} \frac{\E T(1, -n(x+iy))}{n} = \lim_{n \to \infty} \frac{\E T(1, -ni(x+iy))}{n}, 
\end{align*}
which is precisely the statement that the limit norm has quarter-turn symmetry.

A similar restriction arises in any virtually nilpotent group. 
As in Section \ref{defnsandbackground}, let $\Gamma$ be a finitely generated virtually nilpotent group, and let $N$ be a 
torsion-free nilpotent normal subgroup of finite index (for the construction of such a subgroup see Appendix \ref{grouptheory}). 
The conjugation action of $\Gamma$
on $N$ induces an action of $\Gamma/N =: Q$ on $N^{ab}_{free}$. 
It will be convenient later to phrase things in terms of the
right conjugation action, and so we think of the action as a homomorphism $\phi: Q \to \Aut(N^{ab}_{free})^{op}$.
This further induces a right action of $Q$ on 
 $N^{ab} \otimes \R \cong N^{ab}_{free} \otimes \R \cong \g^{ab}$, which, by abuse of notation, we also denote by 
$\phi: Q \to \Aut(\g^{ab})^{op}$. 
We say that a norm on $\Phi$ on $\g^{ab}$ is \emph{conjugation-invariant} if it is $\phi$-invariant, that is,
\[ \Phi(x^{\phi(q)}) = \Phi(x) \]
for all $x \in N^{ab} \otimes \R, q \in Q$.

\begin{prop}
Let $\Gamma, N, \phi$ be as above. If $T$ is a stationary integrable FPP on $\Gamma$ such that the scaling limit of $\E T$ is a Carnot-Carath\'eodory metric on a nilpotent Lie group $G_{\infty}$, then the norm on $\g^{ab}$ associated to this metric is $\phi$-invariant.
\end{prop}
\begin{proof}
The proof is very similar to our example. First, let $\tilde{Q}$ be a finite set of coset representatives of $N$, that is, a finite subset 
$\tilde{Q} \subset \Gamma$
such that the quotient map $\Gamma \to Q$ induces a bijection $\tilde{Q} \leftrightarrow Q$; we may assume without loss of generality that
$\tilde{Q}$ is symmetric (if $\tilde{q} \in \tilde{Q}$, then $\tilde{q}^{-1} \in \tilde{Q}$). Since $\tilde{Q}$ is finite and the FPP
is integrable, there exists some constant $C < \infty$ such that $\E(T(1,\tilde{q}) \le C$ for all $\tilde{q} \in \tilde{Q}$. Then, for any $x \in N$
and any $\tilde{q} \in \tilde{Q}$,
\[ \E T(1, x^{\tilde{q}}) \le \E T(1,\tilde{q}^{-1}) + \E T(1,x) + \E T(1, \tilde{q}) \le \E T(1, x) + 2C \]
where we have used the fact that $\E T$ is left-invariant. Similarly, we have
\[ \E T(1, x) = \E T(1, (x^{\tilde{q}})^{\tilde{q}^{-1}}) \le \E T(1, x^{\tilde{q}}) + 2C, \]
and thus
\[ | \E T(1,x) - \E T(1,x^{\tilde{q}}) | \le 2 C. \]
Since $\phi$ respects the quotient map $N \to N^{ab}_{free}$, 
taking infima over $x \in N$ such that $x^{ab}_{free} = z$ for some fixed $z \in N^{ab}_{free}$ gives
\[ | \tilde{T}(z) - \tilde{T}(z^{\phi(q)}) | \le 2 C = o(z); \]
that is, $\tilde{T}$ is asymptotically equivalent to $\tilde{T}^{\phi(q)}$ for all $q \in Q$, and hence the norm $\Phi$ it induces on $\g^{ab}$
is $\phi(q)$-invariant. 
Pansu's theorem ~\cite{Pansu} tells us that $\Phi$ is the norm in the Carnot-Carath\'eodory construction of the scaling limit
of $(\Gamma, \E T)$, so we are done.
\end{proof}

Although there is certainly more work to be done in exploring necessary conditions for the existence of a limit shape, in all cases which we know
how to prove (~\cite{BenjaminiTessera}, ~\cite{CantrellFurman}), the scaling limit of the random space $(\Gamma, T)$ coincides
with the scaling limit of its mean $(\Gamma, \E T)$, so this tells us that conjugation invariance is a necessary feature of a limit shape
at least in all cases in which we can prove there is a scaling limit.

Theorem \ref{mainthm} then states that this is the \emph{only} obstruction to a Carnot-Carath\'eodory metric on $G_{\infty}$ 
being the limit shape of a stationary FPP on $\Gamma$; that is, as long as the Carnot-Carath\'eodory metric comes from a norm
which is conjugation-invariant, it is the scaling limit of some FPP with stationary weights.

\section{Construction of the edge weights in the virtually nilpotent case} \label{generalconstruction}

Transferring our theorem to the general case is far from automatic, essentially since our Cayley graph may not be nice with respect to the 
the finite index subgroups we wish to pass to. Moreover, instead of keeping track of ``displacements'' of paths by looking at the projection to
$\Gamma^{ab}$, we want to instead look at $N^{ab}_{free}$, and there is typically no nice homomorphism from $\Gamma$ to $N^{ab}_{free}$. 
Nor is there a nice embedding $N^{ab} \to \Gamma^{ab}$; the natural map can have very large kernel (e.g. in our example 
$\Gamma := \langle \rho \rangle \ltimes \Z[i]$ above, $\Gamma^{ab}$ is finite, while $N = N^{ab} = \Z[i]$).
Ultimately, we resolve this by looking at a slightly nonabelian notion of ``displacement'' via the projection $\Gamma \to \Gamma/\widetilde{[N,N]}$,
where we define $\widetilde{[N,N]}$ to be the kernel of the projection $N \to N^{ab}_{free}$.
Note that $\Gamma/\widetilde{[N,N]}$ contains $N^{ab}_{free}$ as a subgroup of finite index.

In spite of these complications, the spirit of the proof exactly the same. Heuristically, we want to ensure that
every direction has the correct ``speed'' at large scales, and we do this by sprinkling long ``fast'' paths throughout the graph which travel
at a certain speed in a certain direction; the rest of the edges are ``slow'' so that any long geodesic must largely avoid them.

It is clear from our above proof that the weight $K$ of the slow edges can be as large as we like, as long as it is finite. We use the slowness
of the edges to account for any error in the fast paths--that is, to guard against the fact that a subpath of a fast path might not 
go in exactly the right direction or exactly at the right speed.

In our first proof, we used the existence of nice paths (Proposition \ref{pathsprop}) which had the property that they (1) stayed close
to the straight line through $b_n$, and (2) traveled ``monotonically forward'' along $b_n$. In the general case, we will want to find
nice paths in $\Gamma/\widetilde{[N,N]}$ which satisfy these properties in a certain ``coarse'' sense to be described below.

Let us now go into more detail understanding the group $\Gamma/\widetilde{[N,N]}$, especially considering it as a finite extension of $N^{ab}_{free}$.
First, take a finite set of coset representatives $\tilde{Q} \subset \Gamma/\widetilde{[N,N]}$ for $N/\widetilde{[N,N]}$; 
we assume for convenience that
$\tilde{Q}$ is symmetric and contains the identity. 
The quotient map $\Gamma/\widetilde{[N,N]} \to Q := (\Gamma/\widetilde{[N,N]})/(N/\widetilde{[N,N]}) 
\cong \Gamma/N$ induces a bijection $\tilde{Q} \to Q$, and we denote its inverse by $s: Q \to \tilde{Q}$. 
If $s$ were a homomorphism, we would have a semidirect
product, but this is not always possible in general. In general, define a function $\eta : Q \times Q \to N^{ab}_{free}$ satisfying
\[ s(q_1) s(q_2) = s(q_1 q_2) \eta(q_1, q_2). \]
This then allows us to understand $\Gamma/\widetilde{[N,N]}$ more explicitly thus: note that 
$Q \times N^{ab}_{free} \to \Gamma/\widetilde{[N,N]}, \\ {(q,n) \mapsto s(q)n}$
is a bijection. Pulling back the multiplication from $\Gamma/\widetilde{[N,N]}$ to the set $Q \times N^{ab}_{free}$ then gives the multiplication
\begin{align*}
(Q \times N^{ab}_{free}) &\times (Q \times N^{ab}_{free}) \to Q \times N^{ab}_{free} \\
(q_1, n_1) &\cdot (q_2, n_2) := (q_1 q_2, \eta(q_1, q_2) + n_1^{\phi(q_2)} + n_2).
\end{align*}
Thus, $\Gamma/\widetilde{[N,N]}$ looks like a semidirect product up to the ``finite error'' introduced by $\eta$.
\begin{rmk}
$\eta$ is in fact a cocycle; the cocycle condition comes precisely from the associativity of the above multiplication. However, we will
not use this fact. Rather, we will repeatedly use the simple fact that $\eta$ is a map from the finite set $Q \times Q$, and thus
has finite image and hence uniformly bounded image.
\end{rmk}
\begin{rmk}
The cocycle $\eta$ of course depends on our choice of $\tilde{Q}$, and the choice is very non-unique.
\end{rmk}

We will now introduce two modified notions of displacement which will be convenient for us. Let $\gamma$ be a path in $E$ 
(the Cayley graph of $\Gamma$) starting at $x \in \Gamma$ and ending at $y \in \Gamma$. We define
\[ \tilde{D}(\gamma) := \bar{x}^{-1} \bar{y} \in \Gamma/\widetilde{[N,N]}, \]
where $\bar{x},\bar{y}$ are the images of $x,y$ under the projection $\Gamma \to \Gamma/\widetilde{[N,N]}$. Note that $\tilde{D}$ is invariant
with respect to the action of $\Gamma$ on paths in $E$ by left multiplication. Note also that for concatenations of paths 
$\gamma = \alpha * \beta$ we have
\[ \tilde{D}(\gamma) = \tilde{D}(\alpha) \tilde{D}(\beta). \]
It will also be helpful for us to have a notion of displacement which lives in $N^{ab}_{free}$ rather than $\Gamma/\widetilde{[N,N]}$; for this, we take
a particular choice of point in $N^{ab}_{free}$ nearby (in the Cayley graph of $\Gamma/\widetilde{[N,N]}$) to $\tilde{D}(\gamma)$:
\[ D(\gamma) := \tilde{D}(\gamma) \tilde{q}(\gamma)^{-1} \in N^{ab}_{free}, \]
where $\tilde{q}(\gamma)$ is the image of $\tilde{D}(\gamma)$ under the composition $\Gamma/\widetilde{[N,N]} \to Q \xrightarrow{s} \tilde{Q}$; 
put another way, using the identification $\Gamma/\widetilde{[N,N]} \leftrightarrow Q \times N^{ab}_{free}$, if $\tilde{D}(\gamma) = (q,n)$,
then $D(\gamma) = (q,n)(q^{-1},0) = n^{\phi(q)^{-1}}$. Note also that if $\tilde{D}(\gamma) \in N^{ab}_{free}$, 
then $\tilde{D}(\gamma) = D(\gamma)$.

$D(\gamma)$ is convenient because it always lands in $N^{ab}_{free}$, the space we are trying to induce the correct norm on; however,
instead of being additive on paths, using the definition and the concatenation property for $\tilde{D}$, we instead get the 
slightly more complicated equation
\begin{equation} \label{eq:noncommdisplacement}
   D(\alpha \beta) = D(\alpha) + D(\beta)^{\phi(\alpha)} + \eta(\alpha, \beta)^{\phi(\alpha \beta)^{-1}}, 
\end{equation}
where in an abuse of notation, we define $\eta(\alpha, \beta) := \eta(q(\alpha), q(\beta)), \phi(\alpha) := \phi(q(\alpha))$, where
$q(\alpha)$ is the image of $\tilde{D}(\alpha)$ under the quotient map $\Gamma/[N,N] \to Q$.
Iterating the above fact easily gives the following by induction:
\begin{prop} \label{noncommdisplacement}
For any paths $\alpha_1,...,\alpha_N$ in $E$, we have
\[ D(\alpha_1 \cdots \alpha_N) = 
D(\alpha_1) + \sum_{i=1}^{N-1} \left( D(\alpha_{i+1}) 
+ \eta(\alpha_1 \cdots \alpha_i, \alpha_{i+1})^{\phi(\alpha_{i+1})^{-1}} \right)^{\phi(\alpha_1 \cdots \alpha_i)^{-1}} \]
\end{prop}
Thus, although the displacements do not add, besides the twisting of $\phi$ we only accumulated at most one uniformly bounded
error term per path concatenated, which will end up being enough later.

From now on we fix an isomorphism $\g^{ab} \cong \R^{d}$ such that $N^{ab}_{free}$ is identified with $\Z^d \subset \R^{d}$
via the map $N^{ab}_{free} \to N^{ab}_{free} \otimes \R \cong \g^{ab} \cong \R^d$. We will often thus identify $D(\gamma)$
with its image in $\R^d$.

We are now ready to state the properties we want for our ``nice'' paths in $E$ (which will become ``fast'' paths).
\begin{lemma} \label{coarsepaths}
   There exists a constant $C_0' > 0$ depending only on $\Gamma$, $S$, $N$, and $\tilde{Q}$ such that, for any vector $u \in \R^d$ and
   any $n \in \Z_{\ge 0}$ there exists a simple path $\gamma$ in $E$ such that 
   \begin{enumerate}
      \item $\gamma$ starts at $1 \in \Gamma$ and $\| D(\gamma) - 2^n u \|_2 \le C_0'$.
      \item $|\gamma| \lesssim |D(\gamma)| \lesssim 2^n \|u\|_2$.
      \item $\gamma$ stays near the line through $u$: 
               If $\alpha$ is a subpath of $\gamma$ starting at 1, then $\| D(\alpha) - \proj_u D(\alpha) \|_2 \le C_0'$.
      \item $\gamma$ is a finite concatenation of paths $\beta_i$ where for each $i$, $|\beta_i| \le C_0'$, 
               $\| D(\beta')^{\phi(q)} \|_2 \le C_0'$ for all $q \in Q$ and every subpath $\beta'$ of $\beta_i$, and 
               \[\pair{ D(\beta_0 \cdots \beta_{i+1}) - D(\beta_0 \cdots \beta_i) }{ \frac{u}{\|u\|_2} } \ge  \frac{1}{C_0'},\]
               that is, $\gamma$ is ``coarsely monotone.''
   \end{enumerate}
   We also assume that $\max_{q_1,q_2,q_3 \in Q} \|\eta(q_1,q_2)^{\phi(q_3)}\|_2 \le C_0'$.
\end{lemma}
This lemma will be proven in Section \ref{coarsepathssection}.

For now, we define the edge weights, very similarly to the first construction. First, given a Carnot-Carath\'eodory metric with associated
norm $\Phi$ on $\g^{ab}$, let $B \subset \g^{ab} \cong \R^d$ be the unit ball of $\Phi$. Let $\{ b_n \}_{n \ge 0}$ be a countable
dense subset of the boundary of $B$. For each $n$, let $\gamma_n$ be the path given in Lemma \ref{coarsepaths} associated to 
the vector $b_n$ and the natural number $n$. Let $E_n$ be the set of edges in $E$ which share at least one vertex with the path $\gamma_n$.

Pick $h > 0$ small enough so that $B_2(0,h) \subset B$ and then choose $K > 0$ large enough so that 
\[ 
   \max_{f \in S, q, q_1, q_2, q_3 \in Q} \frac{ \| D(f)^{\phi(q)} \|_2 + \| \eta(q_1,q_2)^{\phi(q_3)} \|_2 }{ K - 8 C_0' h^{-1} } \le h.
\]
Then define $\eta_n : E_n \to \R_+$ by
\begin{equation*}
  \eta_n(f) = \begin{cases}
                        \frac{\pair{D( \beta_0 \cdots \beta_i) - D( \beta_0 \cdots \beta_{i-1}) }{ \frac{b_n}{\|b_n\|_2}}}
                        {\|b_n\|_2 |\beta_i|}                   & f \in \beta_i,\\
                        K,                                              & \text{otherwise}.
                     \end{cases}
\end{equation*}
where the $\beta_i$ are the subpaths of $\gamma=\gamma_n$ alluded to in Lemma \ref{coarsepaths} (the dependence of $\beta_i$ 
on $n$ is suppressed in the notation).

Lastly, we superimpose randomly sprinkled translated copies of the $\eta_n$ exactly as in the first construction; that is, define 
$\{Z_x\}_{x \in \Gamma}, \{Y_x\}_{x \in \Gamma}$, $X_f$, $x_f$, and $n_f$ exactly as above and then define $w: E_n \to \R_+$
\begin{equation*}
  w(f) = \begin{cases}
              T_{x_f} \eta_{n_f} (f) & X_f \ne \emptyset \\
              K                               & \text{otherwise}.
            \end{cases}
\end{equation*}
By the same arguments as above, these weights are well-defined, ergodic, and uniformly bounded above.
Moreover, the monotonicity condition in Lemma \ref{coarsepaths} implies that each edge has weight at least
\[
   \min_{b \in B} \frac{1}{C_0'^2 \| b \|_2} > 0,
\]
which is to say that $T$ is bi-Lipschitz to the word metric, and we can apply Theorem \ref{cantrellfurmanthm}.

\section{Proof of Theorem \ref{reducedthm} in the general case}
Once again, the proof that the correct norm is induced on $\g^{ab}$ can be reduced to showing the conditions in Proposition 
\ref{normreductionprop}. The proof of the second condition is the same argument as in the simplified case.
(We construct the desired paths by traveling along the center of $N$ until we reach the first fast path that goes in the correct direction,
and then we travel back along the center of $N$.
 We have the same volume growth estimates that we used above as long as 
we assume $\Gamma$ is not virtually $\Z$. In the virtually $\Z$ case, our limit shapes are norms on $\R$, and since all norms on
$\R$ are scalar multiples of each other, we can achieve any desired norm we like by appropriately scaling the weights of, say,
the deterministic FPP which assigns weight 1 to each edge and gives $T=d$.)

For the first condition of Proposition \ref{normreductionprop}, the spirit of the proof is the same, but we have to deal with more error terms.

First, we consider a fast subpath $\gamma$ of $E$ (that is, a path which does not contain any edges of length $K$), and again we note that
it is (up to translation) a subpath of some $\gamma_n$. We further decompose the path into
\[
   \gamma = \alpha \beta_j \cdots \beta_i \omega,
\]
where the $\beta_i$ are the subpaths alluded to in Lemma \ref{coarsepaths} and $\alpha$ and $\omega$ are subpaths of
$\beta_{j-1}$ and $\beta_{i+1}$ respectively.

Now, by Equation \eqref{eq:noncommdisplacement}, we know that
\[
   D(\beta_j \cdots \beta_i)^{\phi(\beta_0 \cdots \beta_{j-1})} = 
   [D(\beta_0 \cdots \beta_i) - D(\beta_0 \cdots \beta_j)] 
   - \eta( \beta_0 \cdots \beta_{j-1}, \beta_j \cdots \beta_i)^{\phi(\beta_0 \cdots \beta_i)^{-1}}.
\]
We can further decompose $[D(\beta_0 \cdots \beta_i) - D(\beta_0 \cdots \beta_j)]$ into its components parallel to $b_n$ and perpendicular to
$b_n$:
\[
   [D(\beta_0 \cdots \beta_i) - D(\beta_0 \cdots \beta_j)] = 
   [D(\beta_0 \cdots \beta_i) - D(\beta_0 \cdots \beta_j)]_{\parallel} +
   [D(\beta_0 \cdots \beta_i) - D(\beta_0 \cdots \beta_j)]_{\perp}.
\]
Now, by our definition of $\eta_n$ we have
\[ 
   T(\beta_j \cdots \beta_i) = 
   \pm \frac{1}{\|b_n\|_2} \pair{D( \beta_0 \cdots \beta_i) - D( \beta_0 \cdots \beta_{j-1}) }{ \frac{b_n}{\|b_n\|_2}},
\]
where we have used coarse monotonicity of $\gamma$. The possible minus sign comes from the fact that we may be traveling forward or 
backward along $\gamma$ (one may check that, since we chose $\tilde{Q}$ to be symmetric, if $\beta$ is any path in $E$ and
$\bar{\beta}$ is $\beta$ with its orientation
reversed, $D(\bar{\beta}) = -D(\beta)$).
Thus, we have
\[
   \frac{ [D(\beta_0 \cdots \beta_i) - D(\beta_0 \cdots \beta_j)]_{\parallel} }{ T(\beta_j \cdots \beta_i) } = \pm b_n \in B.
\]
Moreover, since $\gamma$ stays near to the line through $b_n$ we have
\[
   \frac{ [D(\beta_0 \cdots \beta_i) - D(\beta_0 \cdots \beta_j)]_{\perp} }{ 2 C_0' h^{-1} } \in B_2(0,h) \subset B,
\]
and by assumptions on $C_0'$ we have
\[
   \frac{-\eta(\beta_0 \cdots \beta_{j-1}, \beta_j \cdots \beta_i) }{ C_0' h^{-1} } \in B_2(0,h) \subset B.
\]
Hence by Lemma \ref{convexlemma}
\[
   \frac{D(\beta_j \cdots \beta_i)^{\phi(\beta_0 \cdots \beta_{j-1})}}{ T(\beta_j \cdots \beta_i) + 3C_0' h^{-1} } \in B,
\]
and then by conjugation-invariance of $B$ we have
\[
   \frac{D(\beta_j \cdots \beta_i)}{ T(\beta_j \cdots \beta_i) + 3C_0' h^{-1}} \in B.
\]
Now, since $\alpha$ and $\omega$ are subpaths of  $\beta_{i-1}$ and $\beta_{j+1}$, we have
\[
   \frac{D(\alpha)}{C_0' h^{-1}} , \frac{D(\omega)}{C_0' h^{-1}} \in B_2(0,h) \subset B,
\]
and hence by Lemma \ref{convexlemma}
\begin{align*}
   \frac{D(\alpha \beta_j \cdots \beta_i \omega)}{ T(\beta_j \cdots \beta_i) + 7 C_0' h^{-1} } 
   = \frac{ D(\alpha) + D(\beta_j \cdots \beta_i)^{\phi(\cdot)} 
   + \eta(\cdot,\cdot)^{\phi(\cdot)} + D(\omega)^{\phi(\cdot)} + \eta(\cdot, \cdot)^{\phi(\cdot)}}
   { C_0' h^{-1} + T(\beta_j \cdots \beta_i) + 3 C_0' h^{-1} + C_0' h^{-1} + C_0' h^{-1} + C_0' h^{-1}} \in B,
\end{align*}
where we have again used conjugation-invariance of $B$.
Moreover, for slow edges $f$, by choice of $K$ we have
\[
   \frac{ D(f) + \eta(\cdot, \cdot)^{\phi(\cdot)} }{ T(f) - 8 C_0' h^{-1} } \in B_2(0,h) \subset B.
\]

Writing an arbitrary path $\gamma$ as a concatenation of fast paths and slow edges and using Propositon \ref{noncommdisplacement} gives
\[
   D(\gamma) = 
   \sum_{f \mbox{ slow edges}} (D(f) + \eta( \cdot, \cdot )^{\phi(\cdot)})^{\phi(\cdot)} 
   + \sum_{\gamma' \mbox{ fast paths}} (D(\gamma') + \eta(\cdot, \cdot)^{\phi(\cdot)})^{\phi(\cdot)},
\]
and so using the above and Lemma \ref{convexlemma} gives
\[
   \frac{D(\gamma)}
   {\sum_{f \mbox{ slow edges}} (T(f) - 8 C_0' h^{-1}) 
   + \sum_{\gamma' \mbox{ fast paths}} (T(\gamma') + 8 C_0' h^{-1})} \in B,
\]
and since there is at most one more fast path than there are slow edges, we conclude
\[ 
   \frac{D(\gamma)}{T(\gamma) + 8 C_0' h^{-1}} \in B.
\]

The rest of the proof is just as in the above argument.

\section{Proof of Lemma \ref{coarsepaths}} \label{coarsepathssection}
To prove the existence of ``nice paths'' we want to approximate the nice paths in $\Z^d \cong N^{ab}_{free}$ from Proposition \ref{pathsprop}
and prove that our approximation retains the nice properties ``coarsely''.

First, we prove a lemma which will help control error terms:
\begin{lemma} \label{displacementedgelength}
There exists a constant $K'$ such that for any paths $\alpha, \beta$ in $E$, we have
\[
   \| D(\alpha \beta) - D(\alpha) \|_2 \le K' |\beta|.
\]
\end{lemma}
\begin{proof}
By Equation \eqref{eq:noncommdisplacement}, we know that 
\[
   D(\alpha \beta) - D(\alpha) = D(\beta)^{\phi(\alpha)} + \eta(\alpha, \beta)^{\phi(\alpha \beta)^{-1}}.
\]
First, since the image of $Q$ in $\Aut(N^{ab}_{free}) \cong SL_d^{\pm}(\Z)$ is a finite family of bounded operators on $\R^d$, there is
some constant $M < \infty$ such that
\[
   \| v^{\phi(q)} \|_2 \le M \|v\|_2   
\]
for all $q \in Q, v \in \R^d$. Thus we have $\| D(\beta)^{\phi(\alpha)} \|_2 \le M \| D(\beta) \|_2$.

Next, since $N^{ab}_{free}$ is finite index in $\Gamma/\widetilde{[N,N]}$, it is undistorted,  
which is to say that any word metric on
$N^{ab}_{free}$ is bi-Lipschitz to the restriction to $N^{ab}_{free}$ of any word metric on $\Gamma/\widetilde{[N,N]}$.
(This can be seen using Schreier generators for $N^{ab}_{free}$, see e.g. Theorem 14.3.1 in ~\cite{KarMer}). In particular,
this means that the Euclidean norm $\| \cdot \|_2$ on $N^{ab}_{free}$ is bi-Lipschitz to the metric induced by the Cayley graph
on $\Gamma/\widetilde{[N,N]}$. Hence 
\[
   \| D(\beta) \|_2 \le K'' | D(\beta) | = K'' |\tilde{D}(\beta) \tilde{q}(\beta)^{-1}| \le K''( |\beta| + \max_{\tilde{q} \in \tilde{Q}} |\tilde{q}| ).
\]

Lastly, since $Q$ is finite, we have a uniform bound on the norm of the second term, that is,
\[
   \max_{q_1,q_2,q_3 \in Q} \| \eta(q_1,q_2)^{\phi(q_3)} \|_2 < \infty.
\]
Putting everything together gives
\[
   \| D(\alpha \beta) - D(\alpha) \|_2 \le MK'' |\beta| + \mathrm{const.},
\]
and since every nonempty $\beta$ has $|\beta| \ge 1$ we can easily adjust to get a finite $K'$ which satisfies the desired inequality.
\end{proof}

Now, we construct the paths.
Given $u$ and $n$, first consider the path $\gamma_n$ in $\Z^d \cong N^{ab}_{free}$ using the standard generators $e_i$ of $\Z^d$
given by Proposition \ref{pathsprop}. Next, for each edge $e$ of the path in the standard generators, choose a path $\beta'$ in the Cayley graph for 
$\Gamma/\widetilde{[N,N]}$ induced by the image of $S$ which starts one vertex of $e$ and ends at the other; pick these paths to satisfy
\begin{equation} \label{eq:segmentbound}
   |\beta'| \le \max_{i=1,...,d} d'(1,e_i) =: C
\end{equation}
where $d'$ is the word metric on $\Gamma/\widetilde{[N,N]}$ induced by the image of $S$.
We then lift to a path $\tilde{\beta}'_0 \cdots \tilde{\beta}'_{N-1}$ in $E$. Note that by the properties guaranteed by Proposition \ref{pathsprop} 
we have that: 
\begin{equation} \label{eq:hitstarget}
   \| D(\tilde{\beta}'_0 \cdots \tilde{\beta}'_{N-1}) - 2^n u) \|_2 \le \frac{\sqrt{d}}{2},
\end{equation}
\begin{equation} \label{eq:nottoolong}
   |\tilde{\beta}'_0 \cdots \tilde{\beta}'_{N-1}| \lesssim 2^n \|u\|_2,
\end{equation}
and
\begin{equation} \label{eq:nearparallel}
   \| D(\tilde{\beta}'_0 \cdots \tilde{\beta}'_i) - \proj_u D(\tilde{\beta}'_0 \cdots \tilde{\beta}'_i) \|_2 \le C_0
\end{equation}
for all $i$. If $\alpha$ is a general subpath of $\tilde{\beta}'_0 \cdots \tilde{\beta}'_{N-1}$ starting at $1$, it is of the form 
$\alpha = \tilde{\beta}'_0 \cdots \tilde{\beta}'_i \alpha'$ where $\alpha'$ is a subpath of $\tilde{\beta}'_{i+1}$, and hence
combining Lemma \ref{displacementedgelength} together with Equations \eqref{eq:segmentbound} and \eqref{eq:nearparallel} gives
\begin{equation} \label{eq:stillnearparallel}
   \| D(\alpha) - \proj_u D(\alpha) \|_2 \le C_0 + K' C.
\end{equation}

% paths diagram
\begin{figure}[t]
   \centering
   \includegraphics{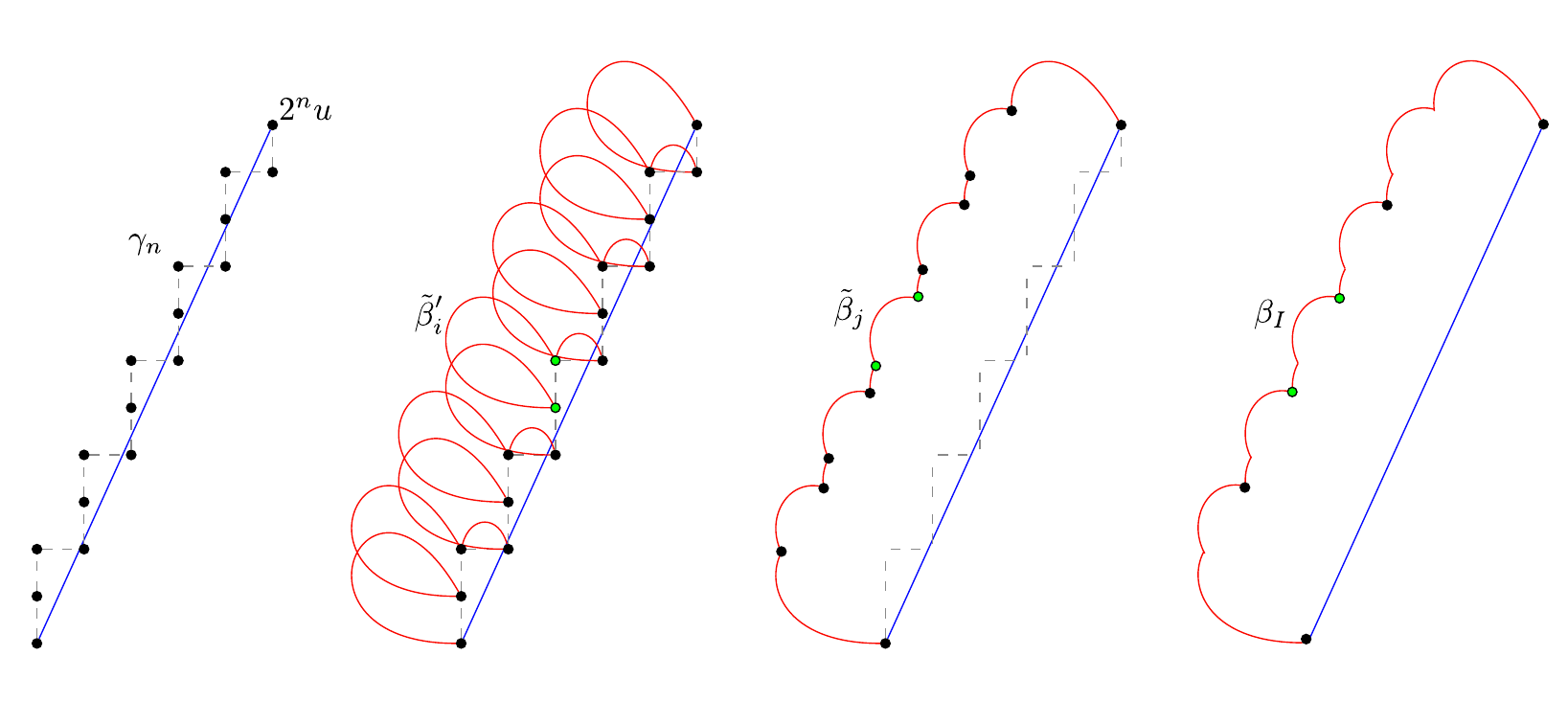}
   \caption{Construction of ``nice paths''. (The ``lifting'' step is omitted here to aid visualization). }
\end{figure}

Thus, $\tilde{\beta}'_0 \cdots \tilde{\beta}'_{N-1}$ satisfies many of the properties we desire. However, it may contain loops, and it may not 
satisfy coarse monotonicity.
So first erase loops to get a 
simple path $\tilde{\beta}_0 \cdots \tilde{\beta}_{N'-1}$. The particular manner in which loops are erased does not matter, so long as the resulting
path is a simple path with the same starting and ending point which is obtained from the original path by deleting subpaths. 
If entire segments $\tilde{\beta}'_i$ are deleted, the number $N'$ of new segments $\tilde{\beta}_0,...,\tilde{\beta}_{N'-1}$ need not be the 
same as $N$ the number of original segments, and some reindexing may be required so that we don't skip indices;
however, every $\tilde{\beta}_i$ is composed of subpaths of a single $\tilde{\beta}'_j$, $j$ depending on $i$.
Thus, each segment $\tilde{\beta}_i$ of the new path still consists of at most $C$ edges. 

Moreover, since the set of displacements of subpaths
of the loop-erased path is a subset of the set of displacements of subpaths of the original path, Equation \eqref{eq:stillnearparallel} 
holds for the new path as well. Equations \eqref{eq:hitstarget} and \eqref{eq:nottoolong} also clearly pass to the loop-erased path as well.

Now we obtain coarse monotonicity.
First we prove the following version of coarse monotonicity for the original Euclidean paths:
\begin{lemma} \label{euclideanmonotonicity}
There exists some $k' > 0$ and $M < \infty$ such that for any $n$, any subpath $\gamma$ of $\gamma_n$ ($\gamma_n$ the path 
in the standard Cayley graph of $\Z^d$ from
Proposition \ref{pathsprop} associated to $n$ and $u$) of length at least $M$ satisfies
\[ \pair{ D(\gamma) }{ \frac{u}{ \|u\|_2 }} \ge k' |\gamma|. \]
\end{lemma}
\begin{proof}
First, we claim that there is a constant $C$ depending only on $d$ such that for any subpath of any $\gamma_n$ of edge-length at
least $C$, at least one edge $f$ of the path satisfies
\[ \pair{D(f)}{\frac{u}{\|u\|_2}} \ge \frac{1}{\sqrt{d}}. \]
Heuristically, this is because the path cannot travel too long in directions perpendicular to $u$ while staying close to the line
through $0$ and $u$.
More rigorously, for some coordinate $i_0 \in \{1,...,d\}$ we have
\[ 
   | \pi_{i_0}(u) | \ge \frac{ \| u \|_2 }{ \sqrt{d} }.
\]
For notational convenience, let's replace some of the standard basis vectors with their opposites
to ensure that $\pair{u}{e_i} = |\pi_i(u)| \ge 0$ for all $i$, and further, let's reindex so that $e_1,...,e_{l}$ satisfy 
$c_i:=\pair{ e_i }{\frac{u}{\|u\|_2}} < \frac{1}{\sqrt{d}}$
and $e_{l+1},...,e_d$ satisfy $c_i \ge \frac{1}{\sqrt{d}}$ for some $0 \le l < d$.

Now let $\gamma$ be a subpath of $\gamma_n$ starting at $x \in \Z^d$ and ending at $y \in \Z^d$, and assume that
for every edge $f$ in $\gamma$,
\[ \pair{D(f)}{\frac{u}{\|u\|_2}} < \frac{1}{\sqrt{d}}. \] 
By Proposition \ref{pathsprop},
$x$ and $y$ must be within Euclidean distance $C_0$ of the line $L$ passing through $0$ and $b_n$ in $\R^n$. Moreover, since we only
travel in directions with low weights, we have $y = x + n_1 e_1 + \cdots + n_l e_l$ for some positive integers $n_i$. 
Now, the distance from $y$ to $L$ is
\begin{align*}
   \mathrm{dist}(y,L) &= \left\| x + n_1 e_1 + \cdots + n_l e_l -\pair{x + n_1 e_1 +\cdots +n_l e_l}{\frac{u}{\|u\|_2}} \frac{u}{\|u\|_2} \right\|_2 \\
   &\ge    \left\| n_1 e_1 + \cdots + n_l e_l - \pair{n_1 e_1 + \cdots + n_l e_l }{\frac{u}{\|u\|_2}} \frac{u}{\|u\|_2} \right\|_2 - \mathrm{dist}(x,L),
\end{align*}
so, since both distances are less than $C_0$, we have 
\begin{align*}
   2C_0 &\ge \left\| n_1 e_1 + \cdots + n_l e_l - \pair{n_1 e_1 + \cdots + n_l e_l }{\frac{u}{\|u\|_2}} \frac{u}{\|u\|_2} \right\|_2 \\
             &= \left\| n_1 e_1 + \cdots n_l e_l - (n_1 c_1 + \cdots n_l c_l)\left(\sum_{i=1}^d c_i e_i \right) \right\|_2 \\
             &\ge \sqrt{ \sum_{i=1}^l \big( n_i - (n_1 c_1 + \cdots + n_l c_l) c_i \big)^2} \\
             &\ge C' \sum_{i=1}^l \big( n_i - (n_1 c_1 + \cdots + n_l c_l) c_i \big) \\
             &\ge C' \sum_{i=1}^l \left( n_i - (n_1 + \cdots + n_l)\frac{1}{d} \right) = C' \left(1-\frac{l}{d}\right) \sum_{i=1}^l n_i \\
             &\ge \frac{C'}{d} \sum_{i=1}^l n_i = \frac{C'}{d} |\gamma|.
\end{align*}
To go from the third to the fourth line, we used that the Euclidean norm is equivalent to the $\ell_1$ norm on $\R^d$,
to go from the fourth to the fifth line, 
we used that $0<c_i<\frac{1}{\sqrt{d}}$ for $i=1,...l$, and to get to the final line we used that $l \le d-1$.

Thus, any subpath of $\gamma_n$ which consists of at least $C := \lfloor\frac{2 C_0 d}{ C' }\rfloor + 1$ edges contains at least one edge
with displacement at least $1/\sqrt{d}$ in the $u$ direction.

Finally, this implies that, for any subpath $\gamma$ of $\gamma_n$ with length at least $2C$ we have
\[
   \pair{D(\gamma)}{\frac{u}{\|u\|_2}} \ge k \lfloor \frac{ |\gamma| }{ C } \rfloor \ge \frac{ k }{2C} |\gamma|.
\]
That is, we have the lemma  with $M = 2C$ and $k' = \frac{k}{2C}$.

\end{proof}

Now take $M' = \max(M, \left\lceil \frac{2K'C + 1}{k'} \right\rceil)$. We then define a new segmentation 
$\beta_0,...,\beta_{\lfloor N'/M' \rfloor - 1}$ of the path by
\[
   \beta_i = \tilde{\beta}_{M'i} \tilde{\beta}_{M'i + 1} \cdots \tilde{\beta}_{M'i + (M'-1)}
\]
if $i < \lfloor N'/M' \rfloor - 1$ and
\[
   \beta_i = \tilde{\beta}_{M'i} \cdots \tilde{\beta}_{N' - 1}
\]
if $i = \lfloor N'/M' \rfloor - 1$.
Note that we have
\[
   |\beta_i| \le 2M' C.
\]
To show that this segmentation of the path gives coarse monotonicity, we have to compare with the original path before erasing loops.
To this end, for a given $i < \lfloor N'/M' \rfloor - 1$, let $I$ be such that $\tilde{\beta}_{(M' + 1)i}$ is a subpath of $\tilde{\beta}'_I$;
that is, the index such that the next edge in $\beta_1 \cdots \beta_{\lfloor N'/M' \rfloor - 1}$ after the segment $\beta_i$
lies in $\tilde{\beta}'_I$. For $i = \lfloor N'/M' \rfloor - 1$, we set $I = N$.
We also set $J$ to be such that the last edge in the path $\beta_{i-1}$ lies in $\tilde{\beta}'_J$; that is, $\tilde{\beta}_{(i-1)M' - 1}$
is a subpath of $\tilde{\beta}'_I$. If $i=0$, we set $J=0$.

Now note that there exists some (possibly empty) subpath $\alpha$ of $\tilde{\beta}'_J$ such that 
\[
   D(\beta_0 \cdots \beta_{i-1} \alpha) = D(\tilde{\beta}'_0 \cdots \tilde{\beta}'_J)
\]
and there exists some subpath $\omega$ of $\tilde{\beta}'_I$ such that
\[
   D(\beta_0 \cdots \beta_i \omega) = D(\tilde{\beta}'_0 \cdots \tilde{\beta}'_I).
\]
Hence, by Lemma \ref{displacementedgelength} and Equation \eqref{eq:segmentbound}, we have that
\begin{equation} \label{eq:perturbedmonotone}
   \| D(\beta_0 \cdots \beta_i) - D(\tilde{\beta}'_0 \cdots \tilde{\beta}'_I) \|_2,
   \| D(\beta_0 \cdots \beta_{i-1}) - D(\tilde{\beta}'_0 \cdots \tilde{\beta}'_J) \|_2 \le K'C,
\end{equation}
which then implies that
\[
   \pair{ D(\beta_0 \cdots \beta_i) - D(\beta_0 \cdots \beta_{i-1}) }{\frac{u}{\|u\|}} 
   \ge \pair{ D(\tilde{\beta}'_0 \cdots \tilde{\beta}'_I) - D(\tilde{\beta}'_0 \cdots \tilde{\beta}'_J) }{ \frac{u}{\|u\|_2} } - 2K'C.
\]
Now, by construction each $\tilde{D}(\tilde{\beta}'_0 \cdots \tilde{\beta}'_i) \in N^{ab}_{free}$, and hence we have
\[
   D(\tilde{\beta'}_0 \cdots \tilde{\beta'}_I) - D(\tilde{\beta}'_0 \cdots \tilde{\beta}'_J) = D(\tilde{\beta}'_{J+1} \cdots \tilde{\beta}'_I),
\]
and then since $D(\tilde{\beta}'_{J+1} \cdots \tilde{\beta}_I)$ is the displacement of a subpath of the path $\gamma_n$ 
(in the \emph{standard} Cayley graph of $\Z^d$)
with edge length at least $I - (J+1) \ge M' \ge M$, Lemma \ref{euclideanmonotonicity} then gives
\[
   \pair{ D(\tilde{\beta}'_0 \cdots \tilde{\beta}'_I) - D(\tilde{\beta}'_0 \cdots \tilde{\beta}'_J) }{ \frac{u}{\|u\|_2} }
   \ge k'M' \ge 2K'C + 1,
\]
and so combining with Equation \eqref{eq:perturbedmonotone} gives
\[
   \pair{ D(\beta_0 \cdots \beta_i) - D(\beta_0 \cdots \beta_{i-1}) }{\frac{u}{\|u\|}} \ge 2K'C + 1 - 2K'C = 1.
\]

Thus, taking
\[
   C_0' := \max\left( \sqrt{d}/2, C_0 + K'C, 2M'C, 1, \max_{q_1,q_2,q_3 \in Q} \eta(q_1,q_2)^{\phi(q_3)}\right)
\]
and $\gamma := \beta_0 \cdots \beta_{\lfloor N'/M' \rfloor - 1}$
gives the Lemma as desired. \qed

\appendix
\appendixpage

\section{ Carnot-Carath\'eodory metrics and the associated graded Lie group} \label{liegroupconstructions}

 In this section we explain the construction needed to describe continuum limits of nilpotent groups, i.e.
 the associated graded nilpotent Lie group associated to a finitely generated virtually nilpotent group, and Carnot-Carath\'eodory
 metrics on this group.
 As above, let $\Gamma$ be a finitely generated virtually nilpotent group, and let $N$ be a torsion-free nilpotent group of finite index.
 A theorem of Mal'cev (~\cite{Malcev}, see also Theorem 2.18 in ~\cite{Rag}) says that there exists a 
 simply connected nilpotent Lie group $G$ such that $N$ is
(isomorphic to) a cocompact lattice in $G$. 
Let $\g$ be the Lie algebra of $G$. Let $\g_{\infty}$ be the associated graded nilpotent Lie algebra, that is
 \[ \g_{\infty} := \bigoplus_{i \ge 1} \g^i / \g^{i+1}, \]
where $\g^1 := \g$, $\g^{i+1} := [\g^i, \g]$ is the descending central series for $\g$.
 Let $G_{\infty}$ be the unique simply connected Lie group which has $\g_{\infty}$ as its Lie algebra. We will refer to $G_{\infty}$
 as the \emph{graded nilpotent Lie group associated to} $\Gamma$.
 
 The map 
 \[ N \xhookrightarrow{} G \to G/[G,G] \cong \g/[\g,\g] =: \g^{ab} \]
 induces an inclusion $N^{ab}_{free} \to \g^{ab}$ and an isomorphism $N^{ab} \otimes \R \to \g^{ab}$.
 Now consider a norm $\Psi$ on $N^{ab} \otimes \R \cong \g^{ab}$.
 Note that $\g^{ab} = \g/[\g,\g]=\g^1/\g^2$ is a vector subspace of $\g_{\infty}$. 
 By left translation
 in $G_{\infty}$, the subspace $\g^{ab} \subset \g_{\infty}$ gives a left-invariant distribution on $TG_{\infty}$,
 and we can extend the norm to any vector in the distribution. Let us call a path 
 $\xi:[a,b] \to G_{\infty}$ \emph{admissible} if it is differentiable a.e. and a.e. $\xi'$ belongs to the support of the distribution. 
 We can then define the $\Psi$-length of $\xi$ to be
 \[ \Psi(\xi) := \int_a^b \Psi(\xi'(t)) dt, \]
 and this gives a metric on $G_{\infty}$ by
 \begin{equation}\label{met:CC}
  d_{\Psi}(x,y) := \inf \{ \Psi(\xi) : \xi \mbox{ is an admissible path from } x \mbox{ to } y \}. 
  \end{equation}
 The metric $d_{\Psi}$ is called the Carnot-Carath\'eodory metric on $G_{\infty}$ associated to $\Psi$.
 Since $\g^{ab}$ generates $\g_{\infty}$ as a Lie algebra, by Chow's theorem ~\cite{GromovCC},
 the topology induced on $G_{\infty}$ by $d_{\Psi}$ coincides with the usual topology on $G_{\infty}$.
 
 The above information is sufficient to understand the statement of the main theorem. The following further data is 
 required to understand Appendix \ref{fillgap}.
 The Lie algebra $\g_{\infty}$ has a one-parameter family of automorphisms $\delta_t : \g_{\infty} \to \g_{\infty}, t > 0$ given by setting
 \[ \delta_t(X) = t^i X\]
 if $X \in \g^i / \g^{i+1}$ and extending by linearity. This of course integrates to a 1-parameter family of automorphisms of $G_{\infty}$,
 which we also denote by $\delta_t$. We refer to $\delta_t$ as \emph{dilations}.
 
 Note that $d_{\Psi}$ is \emph{homogeneous} in the sense that $d_{\Psi}(\delta_t(x), \delta_t(y)) = t d_{\Psi}$.
 In the abelian case, $\Gamma = \Z^d$, $G_{\infty} = \R^d$, the dilations are scalar multiplication by $t$,
 and $d_{\Psi}$ is the usual metric induced by the norm $\Psi$ on $\R^d$.
 
 We now describe a sequence of maps $\Gamma \to G_{\infty}$ which will be Gromov-Hausdorff approximations 
 (see Appendix \ref{fillgap})
 when $\Gamma$ and 
 $G_{\infty}$ are endowed with the appropriate metrics. First, choose a collection of linear subspaces $V_1,...,V_k$ of $\g$ such that for each $i$
 \[ \g = V_1 \oplus \cdots \oplus V_i \oplus \g^{i+1}. \]
 Note that for each $i$, $V_i \subset \g^i$ and the natural map $V_i \to \g^i/\g^{i+1}$ is in isomorphism of vector spaces. Let
 \[ L : \g = V_1 \oplus \cdots \oplus V_k \to \oplus_{i = 1}^k \g^i/\g^{i+1} = \g_{\infty} \]
 be the associated linear isomorphism. Then we define a family of maps
 \[ \scl_t : \Gamma \hookrightarrow G \xrightarrow{\log} \g \xrightarrow{L} \g_{\infty} \xrightarrow{\delta_t} \g_{\infty} \xrightarrow{\exp} G_{\infty}. \]
 (Here $\log$ is the inverse of $\exp: \g \to G$, which is a diffeomorphism, since $G$ is a simply connected nilpotent Lie group).

\section{Understanding the limit norm $\Phi$ via $N^{ab}_{free}$} \label{grouptheory}

Our description of the construction of the limit norm $\Phi$ on $\g^{ab}$ differs slightly from the description in ~\cite{CantrellFurman}.
The two descriptions certainly coincide  in the case that $\Gamma=N$ is a torsion-free finitely generated nilpotent group with 
torsion-free abelianization. However, it's not immediately obvious that their description matches ours in the general
virtually nilpotent case. This section is primarily intended to show how our statement of Theorem \ref{cantrellfurmanthm} follows
from the following:
\begin{thm}{~\cite{CantrellFurman}}
Let $H$ be a finitely generated nilpotent group which is torsion-free and has torsion-free abelianization. Let $T$ be a
stationary random metric on $H$ which is \emph{inner} (see below) and bi-Lipschitz to a word metric on $H$.
Let $d_{\Phi}$ be the Carnot-Carath\'eodory
 metric on $G_{\infty}$ associated to the metric $\E T$, as in Section \ref{defnsandbackground} (with $\Gamma=N=H$). Then almost surely
 \[ (H, \frac{1}{n} T, 1) \tendsto{n}{\infty} (G_{\infty}, d_{\Phi}, 1) \]
 is the sense of pointed Gromov-Hausdorff convergence.
\end{thm}

First let us construct relevant finite-index subgroups. 
Let $\Gamma$ be a finitely generated virtually nilpotent group. 
Then by definition, it contains a nilpotent subgroup $\Gamma'$ of finite index, and this 
is also finitely generated by Schreier's lemma (see e.g. ~\cite{KarMer} Theorem 14.3.1). 
Then $\Gamma'$ contains a \emph{torsion-free} subgroup $\Gamma''$ of finite index (see ~\cite{KarMer}, Theorem 17.2.2).
Take $N$ to be the kernel of the map $\Gamma \to \mathrm{Sym(\Gamma/\Gamma'')}$ given by the action of $\Gamma$
on the cosets of $\Gamma''$ by left multiplication.
Since $N \le \Gamma''$, $N$ is nilpotent and torsion free, and since $N$ is the kernel of a map to a finite subgroup,
it is a finite index normal subgroup of $\Gamma$.

Now we want to extract a finite index subgroup $H$ of $\Gamma$ which is nilpotent, torsion-free,
\emph{and} has torsion-free abelianization. One explicit construction is given by Yves Cornulier in the MathOverflow post ~\cite{YCor};
this construction also has the advantage that that the natural map $H^{ab} \to N^{ab}$ induced by the inclusion $H \xhookrightarrow{} N$
is itself an inclusion (also of finite index).

Here is the construction: recall that we have a projection map $N \to N^{ab} \to N^{ab}/N^{ab}_{tor} =: N^{ab}_{free}$.
Take a basis of $d$ generators $e_1,...,e_d$ for $\Z^d \cong N^{ab}_{free}$, and lift them to $s_1,...,s_d \in N$; then we claim that
$H:= \langle s_1,...,s_d \rangle \le N$ is a finite index subgroup with torsion free abelianization. 

To see that $H^{ab}$ has torsion-free abelianization, consider the natural map $H^{ab} \to N^{ab}_{free}$ induced by the map 
$H \hookrightarrow{} N \to N^{ab}_{free}$.
We claim this is an injection. For if $n_1 \bar{s}_1 + \cdots + n_d \bar{s}_d$ is in the kernel of this map, by the choice of $s_1,...,s_d$ this
means that $n_1 e_1 + \cdots + n_d e_d = 0$, which implies that $n_1,...,n_d=0$, since $e_1,...,e_d$ is a basis.
The map is also clearly surjective by construction, so $H^{ab} \cong N^{ab}_{free}$ and so $H$ has torsion-free abelianization.

To see that $H$ is finite index, first note that, from the above, $H^{ab} \le N^{ab}$ is finite index. We then use the following lemma; the proof
is taken from Cornulier's argument in ~\cite{YCor}:

\begin{lemma} 
Let $N$ be a finitely generated nilpotent group, and let $H$ be subgroup of $N$ such that $H[N,N]$ is finite index in $N$ (equivalently,
$H^{ab} \to N^{ab}$ has finite-index image in $N^{ab}$).
Then $H$ is finite index in $N$.
\end{lemma}
\begin{proof}
We proceed by induction on the nilpotency degree of $N$. If $N$ is abelian, then the statement is immediate.

Suppose the statement holds for all nilpotent groups of degree $k-1$, and suppose $N$ is degree $k$. Let $N^k$ be the $k^{th}$ subgroup
in the descending central series for $N$. By our inductive hypothesis applied to $N/N^k$, $HN^k$ is a finite index subgroup of $N$.
So all that remains is to show that $H$ is finite index in $HN^k$.

For this, first note that since all $(k+1)$-fold commutators vanish, the $k$-fold commutator map $N \times \cdots \times N \to N^k$ is ``multilinear''
in the sense that
\[ [a_1,\cdots,xy,\cdots,a_k] 
= [a_1,\cdots,x,\cdots,a_k]\cdot [a_1,\cdots,y,\cdots,a_k]; \]
we also see that the output only depends on the abelianizations of $a_1,...,a_k$, and thus the $k$-fold commutator map induces a
surjective homomorphism $N^{ab} \otimes \cdots \otimes N^{ab} \to N^k$. We claim that the map $\bigotimes^k H^{ab} \to \bigotimes^k N^{ab}$
induced by the finite index inclusion $H \to N$ has image which is finite index in $\bigotimes^k N^{ab}$. Once we know this,
since $H^k$ is precisely the composition of the map $\bigotimes^k H^{ab} \to \bigotimes^k N^{ab} \to N^k$, $H^k$ is finite index in
$N^k$, and hence $H$ is finite index in $HN^k$.

Now, to see that the image of $\bigotimes^k H^{ab} \to \bigotimes^k N^{ab}$ is finite index, we use the following general fact: If $A$ is
a finitely generated abelian group and $B \le A$ is a subgroup of finite index, then for any $i \ge 1$, $\bigotimes^i B \le \bigotimes^i A$
is finite index. For $i=1$, this is immediate. Now, inductively assume $T'$ is a finite set such that $T' + \bigotimes^i B = \bigotimes^i A$, and let $S'$ be a finite generating set for
$\bigotimes^i B$. Also let $T$ be a finite set such that $T + B = A$ and let $S$ be a finite generating set for $B$.
We claim that the set
\[
   \left\{ \sum_{\sigma \in S'} t_{\sigma} \otimes \sigma 
   + \sum_{\tau \in T'}\left( t_{\tau} \otimes \tau + \sum_{s \in S}  s \otimes t'_{s,\tau} \right) :
   t_{\sigma}, t_{\tau} \in T, t'_{s,\tau} \in T' \right\}
\]
forms a finite set of coset representatives for $\bigotimes^{i+1} B$ in $\bigotimes^{i+1} A$.

To see this, first consider a general element of $\bigotimes^{i+1} A$. It is a sum of elements of the form
\[ 
   (\sum_{s \in S} m_s s + t) \otimes (\sum_{s' \in S'} m_{s'} s' + t')
\]
where $t \in T, t' \in T'$, $m_s, m_{s'} \in \Z$,
and hence, by expansion, equal to 
\[
   \sum_{\sigma \in S'} ( \sum_{s \in S} m_{\sigma, s} s + t_{\sigma} ) \otimes \sigma +
   \sum_{\tau \in T'} ( \sum_{s \in S} m_{\tau, s} s + t_{\tau} ) \otimes \tau
\]
for some $m_{\sigma, s}, m_{\tau, s} \in \Z$, $t_{\sigma}, t_{\tau} \in T$.
Since every $s \otimes \sigma \in \otimes^{k+1} B$, the element
\[
   \sum_{\sigma \in S'} t_{\sigma} \otimes \sigma + \sum_{\tau \in T'}\left( t_{\tau} \otimes \tau + \sum_{s \in S} s \otimes m_{\tau, s} \tau \right)
\]
represents the same coset of $\otimes^{k+1} B$.
For each $s, \tau$, by the inductive hypothesis, we have
\[ 
   s \otimes m_{\tau,s} \tau = s \otimes \left( \sum_{s' \in S'} n_{s', \tau} s' + t'_{s,\tau} \right)
\]
for some $n_{s',\tau} \in \Z$ and $t'_{s,\tau} \in T'$, and this is equivalent \emph{modulo} $\otimes^k B$ to 
\[
   \sum_{s' \in S'} s \otimes t'_{s,\tau}.
\]
That is, an arbitrary element is equivalent to one in the set provided, as desired.
\end{proof}

In sum, we have $H \le N \unlhd \Gamma$ finite index inclusions, where $N$ is torsion-free and $H$ is torsion-free with torsion-free abelianization.

Now, let $T$ be a stationary random metric on $\Gamma$ which is almost surely inner and bi-Lipschitz to a word metric on $\Gamma$.
Recall that a metric space is called \emph{inner} if for all $\epsilon > 0$, there exists $0<R<\infty$ such that
for any $x,y \in \Gamma$, there exists an $(\epsilon,R)$-coarse geodesic from $x$ to $y$,
that is, a sequence $x=p_0, p_1,...,p_M=y$ in $\Gamma$ such that each $d(x_{i-1},x_i) \le R$ and
\[
   \sum_{i=1}^M d(p_{i-1},p_i) \le (1+\epsilon)d(x,y).
\]
(Note that, in the main body of the paper, we consider $T$ an FPP with edge weights $w$ uniformly bounded above;
such $T$ is automatically inner). 
We want to show that
\[
   (\Gamma, \frac{1}{n} T) \to (G_{\infty}, d_{\Phi}).
\]
By Proposition \ref{prop:finiteindex}, it suffices to show that
\[
   (H, \frac{1}{n} T|_H) \to (G_{\infty}, d_{\Phi}).
\]
Thus, we want to apply Theorem \ref{cantrellfurmanthm} to $H$, so first we must check that the hypotheses are satisfied.
\begin{prop}
Let $\Gamma, H, T$ be as above. Then $T|_H$ is bi-Lipschitz to a word metric on $H$ and $T|_H$ is inner.
\end{prop}
\begin{proof}
$T|_H$ is bi-Lipschitz to $d|_H$, and since $H \le \Gamma$ is finite index, any word metric on $H$ is bi-Lipschitz to $d|_H$ (this can be
seen using Schreier generators for $H$, see e.g. Theorem 14.3.1 in ~\cite{KarMer}),
so we have the first claim.

Next, we show innerness. Let $\epsilon > 0$. First, using the innerness of $T$ on $\Gamma$, choose $r>0$ so that any $x,y \in \Gamma$
can be joined by an $(\frac{\epsilon}{2}, r)$-coarse geodesic. Next, note that since $H \le \Gamma$ is finite index and $T \le K d$ a.s. 
for some $K< \infty$, we have
\[
   \max_{g \in \Gamma} T(g,H) \le K \max_{g \in \Gamma} d(g,H) =: C
\]
for some non-random constant $0<C<\infty$.
Now choose $0<R<\infty$ sufficiently large so that $0 < \frac{4C}{R-r} \le \frac{\epsilon}{2}$. We claim that any $h, h' \in H$ can be joined
by an $(\epsilon, R + 2C)$-coarse geodesic in $H$.

To construct such a coarse geodesic, first take an $(\frac{\epsilon}{2}, r)$-coarse geodesic $h=p'_0,p'_1,...,p'_{M'}=h'$ in $\Gamma$.
By deleting points, we can construct a $(\frac{\epsilon}{2},R)$-coarse geodesic $h=p_0,...,p_M=h'$ with 
\[
   M \le \left\lceil \frac{ T(h,h') }{ R - r} \right\rceil \le \frac{ 2T(h,h') }{R-r},
\]
where the last inequality only holds for $T(h,h') \ge R-r$, but if $T(h,h') \le R + 2C$ then $p_0=h, p_1=h'$ trivially gives an 
$(\epsilon, R+2C)$-coarse geodesic, so we may assume this inequality holds.

Lastly, for each $p_i$, choose $q_i \in H$ with $T(p_i,q_i) \le C$ (and of course $q_0=p_0=h, q_M=p_M=h'$). Then
each $T(q_{i-1},q_i) \le T(p_{i-1},p_i) + 2C \le R + 2C$ and 
\begin{align*}
   \sum_{i=1}^M T(q_{i-1},q_i) &\le \sum_{i=1}^M T(p_{i-1},p_i) + 2CM \le (1+\frac{\epsilon}{2}) T(h,h') + 2CM \\
                                               &\le  (1 + \frac{\epsilon}{2})T(h,h') + 2C \cdot \frac{2 T(h,h')}{R - r} \\
                                               &\le (1 + \epsilon)T(h,h'),
\end{align*}
so $q_0,...,q_M$ is an $(\epsilon,R+2C)$-coarse geodesic in $H$, as desired.
\end{proof}

Now, note that the Malcev completions of $H$ and $N$ coincide; if $N$ is a cocompact lattice in $G$, then as a finite-index subgroup of $N$,
$H$ is also cocompact in $G$. Therefore $H$ and $N$ have the same associated graded nilpotent Lie group $G_{\infty}$ as well.
Thus, Theorem \ref{cantrellfurmanthm} tells us that
\[
   (H, \frac{1}{n} T|_H) \to (G_{\infty}, d_{\Phi_H}),
\]
where we define $\Phi_H$ to be the unique norm on $\g^{ab}$ asymptotically equivalent to the subadditive function
\[
   \tilde{T}_H (h) := \inf_{t \in H : t^{ab} = h} \E T(1,t)
\]
on $H^{ab}$. (Recall that we can relate functions on $H^{ab}$ and $\g^{ab}$, since we have
a map $H^{ab} \to \g^{ab}$ and an isomorphism $H^{ab} \otimes \R \cong \g^{ab}$ induced by the composition
\[
   H \hookrightarrow{} G \to G/[G,G] \cong \g/[\g,\g] =: \g^{ab}.)
\]

Thus, to deduce our statement of Theorem \ref{cantrellfurmanthm}, it only remains to show that $\Phi_H = \Phi$, where recall that we define
$\Phi$ to be the unique norm on $\g^{ab}$ which is asymptotically equivalent to the subadditive function
\[
   \tilde{T}(n) := \inf_{t \in N : t^{ab}_{free} = n} \E(1,t)
\]
on $N^{ab}_{free}$. 

\begin{prop}
$\Phi_H = \Phi$.
\end{prop}
\begin{proof}
Note that $H^{ab}$ and $N^{ab}_{free}$ are identified with the same subgroup of $\g^{ab}$ since the 
inclusion $H^{ab} \to \g^{ab}$ is exactly equal to the composition of the isomorphism $H^{ab} \cong N^{ab}_{free}$
and the inclusion $N^{ab}_{free} \to \g^{ab}$. Using the isomorphism $H^{ab} \cong N^{ab}_{free}$ to consider $\tilde{T}_H$
as a subadditive function on $N^{ab}_{free}$, we have
\[
   \tilde{T}_H (n) = \inf_{t \in H : t^{ab}_{free} = n} \E T(1,t).
\]
From this it is clear that $\tilde{T} \le \tilde{T}_H$.

To show a lower bound, first note that
since $H$ is finite index in $N$, $H \cap \widetilde{[N,N]}$ is finite-index in $\widetilde{[N,N]}$.
Let $R$ be a finite set of right coset representatives for $H \cap \widetilde{[N,N]}$ in $\widetilde{[N,N]}$, 
that is, $N \cap \widetilde{[N,N]} = \bigcup_{r \in R} H \cap \widetilde{[N,N]} r$.
Set $ C := \max_{r \in R} |r|$, where $|\cdot| = d(1,\cdot)$ is, as always, the word length in $\Gamma$ with respect to the generating
set $S$. Then we have
\[
   \tilde{T}(n) = \inf_{t \in H, r \in R : t^{ab}_{free} = n} \E T(1,tr) 
   \ge \inf_{t \in H, r \in R : t^{ab}_{free} = n} \E T(1,t) - \E T(1,r) \ge \Phi_H(n) - KC,
\]
where we have used that $T \le Kd$.
Thus $ |\tilde{T}(n) - \tilde{T}_H(n)| \le KC = o(n) $ and $\Phi = \Phi_H$, as desired.
\end{proof}

\section{Gromov-Hausdorff convergence to the limit shape} \label{fillgap}

Recall the notion of pointed Gromov-Hausdorff convergence (~\cite{GromovMetric}). There are many equivalent conditions for this convergence,
but here we use a particular sufficient condition. Let $(X_n,d_n, o_n), (X_0,d_0, o_0)$ be metric spaces with distinguished basepoints
 $o_n, o_0$.
 A sequence of maps $f_n: X_n \to X_0$
 is called a \emph{sequence of of pointed Gromov-Hausdorff approximations} if for every $\epsilon > 0$, for all sufficiently large $n$
 we have
 \begin{enumerate}
 \item $d_0(f_n(o_n,o_0)) < \epsilon$,
 \item every point of $B(o_0, 1/\epsilon)$ is within distance $\epsilon$ of $f_n(B(o_n, 1/\epsilon))$,
 \item $(1-\epsilon)d_n(x,y) - \epsilon \le d_0(f_n(x),f_n(y)) \le (1+\epsilon)d_n(x,y) _ \epsilon$ for all $x,y \in B(o_n, 1/\epsilon)$.
 \end{enumerate}
 If $f_n:X_n \to X_0$ is a sequence of pointed Gromov-Hausdorff approximations, then $X_n$ pointed Gromov-Hausdorff converges to $X_0$.
Here, our metric spaces are groups with various metrics, and the basepoint will always be the identity element.

In ~\cite{CantrellFurman}, Section 4.4, Cantrell and Furman prove the following: for any fixed $g, g' \in G^{\infty}$, almost surely
\begin{equation}
   \lim_{\epsilon \to 0} \limsup_{t \to \infty} \sup \left\{ \frac{1}{t} |T(\gamma, \gamma') - d_{\Phi}(g,g')| : 
   \gamma,\gamma' \in \Gamma, d_{\|\cdot\|}(\scl_{\frac{1}{t}} \gamma, g), d_{\|\cdot\|}(\scl_{\frac{1}{t}} \gamma', g') < \epsilon \right\} = 0,
   \label{eq:ptwise}
\end{equation}
where $\Gamma, G_{\infty}, T, d_{\Phi}, d_{\| \cdot \|}$ are all as defined in Section \ref{defnsandbackground},
and the maps $\scl_{\frac{1}{t}} : N \to G_{\infty}$ are as defined in Appendix \ref{liegroupconstructions}.
In particular,
$(G_{\infty},d_{\| \cdot \|})$ is the scaling limit of $\Gamma$ endowed with the word metric as given by Pansu's theorem:
\begin{thm}{(Pansu, ~\cite{Pansu})} \label{pansuthm}
   \[ \scl_{\frac{1}{t}} : (\Gamma, \frac{1}{t} d) \to (G_{\infty}, d_{\| \cdot \|}) \]
   is a sequence of Gromov-Hausdorff approximations.
\end{thm}

To prove that $\scl_{\frac{1}{t}} : (\Gamma, \frac{1}{t} T) \to (G_{\infty}, d_\Phi)$ is a sequence of Gromov-Hausdorff approximations, 
by homogeneity of the norm $d_{\Phi}$, it suffices 
to show that, for any $\epsilon > 0$, there exists $R>0$ such that for any 
$|\gamma|,|\gamma'| \ge R$,
\[
   | T(\gamma, \gamma') - d_{\Phi}(\scl_1(\gamma), \scl_1(\gamma')) | \le \epsilon \max(|\gamma|,|\gamma'|).
\] 
The rest of this appendix is devoted to proving this fact.
\begin{rmk}\label{rem:correction}
In ~\cite{CantrellFurman}, it is shown that the event of failure of Gromov-Hausdorff convergence
is contained in an uncountable union of null-sets. More specifically, they show that failure of Gromov-Hausdorff convergence
entails the existence of some pair $g,g' \in G_{\infty}$ for which Equation \eqref{eq:ptwise} fails, but a priori $(g,g')$ ranges over
the uncountable set $G_{\infty} \times G_{\infty}$.
It is necessary to show that it is contained in a \emph{countable} union of 
null-sets.
\end{rmk}
Now, let $\{(g_n,g'_n)\}$ be a countable dense subset of $G_{\infty} \times G_{\infty}$. With probability 1, Equation \eqref{eq:ptwise}
holds for all $(g_n,g'_n)$ simultaneously. We show that on this probability 1 subset Gromov-Hausdorff convergence holds.

Suppose that Gromov-Hausdorff convergence fails, that is, there exists $\epsilon_0 > 0$ and some sequence 
$(\gamma_n, \gamma'_n) \in \Gamma \times \Gamma$ with $\min(|\gamma_n|, |\gamma'_n|) \to \infty$ such that
\[
   \frac{1}{t_n} | T(\gamma, \gamma') - d_{\Phi}(\scl_1(\gamma), \scl_1(\gamma') | \ge \epsilon_0,
\]
where we define $t_n := \max( |\gamma_n| , |\gamma'_n| )$.
By homogeneity of $d_{\Phi}$, this is equivalent to
\begin{equation} \label{eq:counterexample}
   \left| \frac{1}{ t_n} T(\gamma, \gamma') - 
   d_{\Phi}(\scl_{\frac{1}{t_n}} \gamma_n, \scl_{\frac{1}{t_n}} \gamma'_n) \right|
   \ge \epsilon_0.
\end{equation}
Since the sequence $(\scl_{\frac{1}{t_n}} \gamma_n, \scl_{\frac{1}{t_n}} \gamma'_n)$
lies in the product of the unit $d_{\|\cdot\|}$ balls of $G_{\infty}$, by compactness we may pass to a subsequence and assume that
\[
   (\scl_{\frac{1}{t_n}} \gamma_n, \scl_{\frac{1}{t_n}} \gamma'_n) \to (g_0,g'_0)
\]
for some $(g_0,g'_0) \in G_{\infty} \times G_{\infty}$. Convergence holds in the $d_{\|\cdot\|}$ metric as well as the $d_{\Phi}$ metric.

Now choose $N$ sufficiently large so that
\begin{equation} \label{eq:discreteapproximationclose}
   |d_{\Phi}( \scl_{\frac{1}{t_n}} \gamma_n, \scl_{\frac{1}{t_n}} \gamma'_n)
   - d_{\Phi}(g_0,g_0')| \le \frac{\epsilon_0}{2}
\end{equation}
for all $n \ge N$.
Combining Equations \eqref{eq:counterexample} and \eqref{eq:discreteapproximationclose} gives
\begin{equation} \label{eq:contradict}
   |\frac{1}{t_n} T(\gamma_n, \gamma'_n) - d_{\Phi}(g_0, g'_0) | \ge \frac{\epsilon_0}{2}.
\end{equation}
Fix $\delta' > 0$ (to be chosen later).
Now choose $(g_{m_0}, g'_{m_0})$ from our countable dense set such that
\[
   \max( d_{\| \cdot \|}(g_{m_0}, g_0), d_{\| \cdot \|}(g'_{m_0}, g'_0), d_{\Phi}(g_{m_0}, g_0), d_{\Phi}(g'_{m_0}, g'_0)) \le \delta'.
\]
For each $k \ge 1$ define $\gamma_{m_0}^k$ to be the $\gamma \in \Gamma$ such that $\scl_{\frac{1}{k}}$ has minimal
distance to $g_{m_0}$, and similarly define $\gamma_{m_0}'^k$.
Then by Equation \eqref{eq:ptwise} we have
\[
   \left| \frac{1}{k} T(\gamma_{m_0}^k, \gamma_{m_0}'^k) - d_{\Phi}(g_{m_0}, g'_{m_0}) \right| \tendsto{k}{\infty} 0,
\]
and so we can choose $N$ also sufficiently large that for all $n \ge N$,
\[
   \left| \frac{1}{t_n} T(\gamma_{m_0}^{t_n}, \gamma_{m_0}'^{t_n}) - d_{\Phi}(g_{m_0}, g'_{m_0}) \right| \le \delta'.
\]
By Theorem \ref{pansuthm} we can also choose $N$ so that for all $n \ge N$,
\[
   \left| \frac{1}{t_n} d(\gamma_n, \gamma_{m_0}^{t_n}) - d_{\| \cdot \|}( g_0, g_{m_0}) \right|
   \le \delta',
\]
\[
   \left| \frac{1}{t_n} d(\gamma'_n, \gamma_{m_0}'^{t_n}) - d_{\| \cdot \|}( g'_0, g'_{m_0} ) \right|
   \le \delta'.
\]
Thus we have (again taking $k=\max(|\gamma_n|, |\gamma'_n|)$)
\begin{align*}
   \left| \frac{1}{t_n} T(\gamma_n, \gamma'_n) - d_{\Phi}(g_0,g'_0) \right| 
   &\le \left| \frac{ T(\gamma_n, \gamma'_n) - T(\gamma_{m_0}^{t_n}, \gamma_{m_0}'^{t_n}) }{ t_n } \right| \\
   &+ \left| \frac{1}{t_n} T(\gamma_{m_0}^{t_n}, \gamma_{m_0}'^{t_n}) - d_{\Phi}(g_{m_0}, g'_{m_0}) \right| \\
   &+ | d_{\Phi}(g_{m_0}, g'_{m_0}) - d_{\Phi}(g_0,g'_0)|.
\end{align*}
By our choice of $(g_{m_0}, g'_{m_0})$, we have that the last term is bounded by $2 \delta$.
If $n \ge N$, we have that the second term is bounded by $\delta$.
To bound the first term, recall that by assumption, $T \le Kd$ and hence
\begin{align*}
   | T(\gamma_n, \gamma'_n) - T(\gamma_{m_0}^{t_n}, \gamma_{m_0}'^{t_n}) |
   \le T(\gamma_n, \gamma_{m_0}^{t_n}) + T(\gamma'_n, \gamma_{m_0}'^{t_n})
   \le K (d(\gamma_n, \gamma_{m_0}^{t_n}) + d(\gamma'_n, \gamma_{m_0}'^{t_n})),
\end{align*}
and so
\begin{align*}
   \left| \frac{T(\gamma_n, \gamma'_n) - T(\gamma_{m_0}^{t_n}, \gamma_{m_0}'^{t_n})}{t_n} \right|
   &\le K\left( \frac{1}{t_n}d(\gamma_n, \gamma_{m_0}^{t_n}) + \frac{1}{t_n}d(\gamma'_n, \gamma_{m_0}'^{t_n})\right) \\
   &\le K\left( d_{\|\cdot\|}(g_0,g_{m_0}) + \delta + d_{\|\cdot\|}(g'_0,g'_{m_0}) + \delta \right) \le 4K\delta.
\end{align*}
All in all we have
\[
   \left| \frac{1}{t_n} T(\gamma_n, \gamma'_n) - d_{\Phi}(g_0,g'_0) \right| \le 4K\delta + 3\delta,
\]
and for a sufficiently small choice of $\delta$, this contradicts Equation \eqref{eq:contradict}, and so we are done.

 \subsection*{Acknowledgements} 
 The authors thank Nir Avni for providing resources on finitely generated nilpotent groups, 
 Sami Douba for conversations on topics related to this paper,
 and Yves Cornulier for providing helpful context by giving a counterexample showing that on nilpotent groups, being
 bi-Lipschitz to a word metric does not imply innerness.

\bibliography{nilfppbib}
\bibliographystyle{plain}
\end{document}